\documentclass[a4paper]{amsart}

\usepackage{amsmath} % per avere \eqref, \align ... e per avere equation*, etichette equazioni, ...
\usepackage{amsfonts} % Per font matematici: \mathbb, ...
\usepackage{amssymb} % Per simboli matematici aggiuntivi: \twoheadleftarrow, \twoheadrightarrow, \leftarrowtail, \rightarrowtail, ...
\usepackage{amsthm} % Per gestione di: teoremi, lemmi, proposizioni, corollari, definizioni, osservazioni, esempi, ...
	\theoremstyle{plain}% default
		\newtheorem{theorem}{Theorem}[section]
		\newtheorem{lemma}[theorem]{Lemma}
		\newtheorem{proposition}[theorem]{Proposition}
		
		\newtheorem{introtheorem}{Theorem}
	\theoremstyle{definition}
		\newtheorem{definition}[theorem]{Definition}
		
		\newtheorem*{conjecture}{Conjecture}
	\theoremstyle{remark}
		\newtheorem{remark}[theorem]{Remark}
		\newtheorem{example}[theorem]{Example}
\usepackage{enumerate} % Per elenchi personalizzati
\usepackage{tikz-cd} % Per diagrammi commutativi

\newcommand{\N}{\mathbb{N}}
\newcommand{\Z}{\mathbb{Z}}

\newcommand{\C}{\mathbb{C}}
\newcommand{\PP}{\mathbb{P}}
\newcommand{\GG}{\mathbb{G}}
\newcommand{\FF}{\mathbb{F}}

\title[Universal vector bundles, push-forward formul{\ae} and positivity]{Universal vector bundles, push-forward formul{\ae} and positivity of characteristic forms}
\author{Filippo Fagioli}
\address{Dipartimento di Matematica e Informatica \textquotedblleft{Ulisse Dini}\textquotedblright\\Università degli Studi di Firenze\\Viale Morgagni 67/a\\50134 Firenze, Italia}
\email{filippo.fagioli@unifi.it}
\thanks{The author is supported by the project PRIN2017 ``Real and Complex Manifolds: Topology, Geometry and holomorphic dynamics'' (code 2017JZ2SW5).}

\date{\today}

\begin{document}

\begin{abstract}
	Given a Hermitian holomorphic vector bundle over a complex manifold, consider its flag bundles with the associated universal vector bundles endowed with the induced metrics.
	We prove that the universal formula for the push-forward of a polynomial in the Chern classes of all the possible universal vector bundles also holds pointwise at the level of Chern forms.
	A key step in our proof is the explicit computation, at a point of any flag bundle, of the Chern curvature of the universal vector bundles with the induced metrics.
	
	As an application, we provide an alternative version of the Jacobi--Trudi identity at the level of differential forms.
	We also show the positivity of a family of polynomials in the Chern forms of Griffiths semipositive vector bundles.
	This latter result partially confirms the Griffiths' conjecture on positive characteristic forms, which has raised considerable interest in recent years.
\end{abstract}

\maketitle

\section*{Introduction}
	Let $ X $ be a complex manifold of dimension $ n $ and let $ E \to X $ be a holomorphic vector bundle of rank $ r \ge 2 $.
	Fix a sequence $ \rho $ of dimensions
	$$ 0 = \rho_0 < \dots < \rho_l < \dots < \rho_{m} = r $$
	and consider the associated flag bundle $ \pi_{\rho} \colon \FF_{\rho}(E) \to X $, which naturally has a filtration of $ m+1 $ tautological vector bundles $ U_{\rho,l} \to \FF_{\rho}(E) $ of rank $ \rho_{l} $.
	Consider also all the possible universal quotient bundles that arise from this filtration.
	Both tautological bundles and their universal quotients are called \emph{universal vector bundles} over $ \FF_{\rho}(E) $.
	Let $ \mathcal{E}_1, \dots, \mathcal{E}_N $ be an enumeration of all the universal vector bundles and denote by $ r_1, \dots, r_N $ the corresponding ranks.
	Finally, for each $ j = 1, \dots, N $ consider the Chern classes $ c_{1}(\mathcal{E}_j) , \dots, c_{r_j}(\mathcal{E}_j) $ in the cohomology of $ \FF_{\rho}(E) $. 
	
	Now, given a homogeneous polynomial $ F $ in $ r_1 + \dots + r_N $ variables of degree $ d_\rho + k $, where $ d_\rho $ denote the relative dimension of the submersion $ \pi_{\rho} $, the push-forward 
	\[
	(\pi_{\rho})_{*} F \bigl( c_{\bullet}(\mathcal{E}_{1}),\dots,c_{\bullet}(\mathcal{E}_{N}) \bigr)
	\] 
	gives a characteristic class of $ E $ in the cohomology group $ H^{2k}(X) $.
	It is therefore natural to try to determine a formula to express this class as some polynomial formally evaluated in the Chern classes of $ E $.
	
	This classical problem has been considered and settled by several authors in different degree of generality, and there is a vast literature concerning push-forward formul{\ae} (which are also called \emph{Gysin formul{\ae}}, see \cite{Ful98}) for flag bundles.
	This paper essentially considers the formul{\ae} given in \cite{DP17} (see also \cite{Ilo78}), however we just remark that the approaches followed are various and, for instance, Gysin formul{\ae} for flag bundles are provided by:
	\begin{itemize}
		\item \cite{Qui69,Dam73,AC87} by using Grothendieck residues;
		\item \cite{JLP81} by using Schur functions;
		\item \cite{Bri96,PR97} by using symmetrizing operators;
		\item \cite{BS12,Tu17,Zie18} by using residues at infinity.
	\end{itemize}
	See also the books \cite{FP98,Man98}.
	We also mention an explicit Gysin formula for Grassmann bundles (associated to sequences of type $ (0,\rho_{1},r) $) provided in \cite{KT15}.
	
	\medskip
	Beside the cohomological situation, it is natural to investigate the analogue at the level of representatives as follows.
	Equip $ E $ with a smooth Hermitian metric $ h $.
	From the Chern curvature tensor $ \Theta(E,h) $, let us consider the corresponding Chern forms on $ X $ defined, for $ 0 \le s \le r $, as
	\[
	c_s(E,h) = \operatorname{tr}_{\operatorname{End}({\Lambda}^{s} E)}\left( {\bigwedge}^s \frac{i}{2\pi} \Theta(E,h) \right).
	\]
	By the Chern--Weil theory the form $ c_s(E,h) $ is $ d $-closed, real and represents the Chern class $ c_s(E) $ of the vector bundle $ E $.
	
	The universal vector bundles $ \mathcal{E}_j $'s considered before inherit, being sub-bundles of $ \pi_{\rho}^{*} E $ or quotients of them, natural Hermitian metrics $ H_j $'s.
	Thus, the classes $ c_{1}(\mathcal{E}_j) , \dots, c_{r_j}(\mathcal{E}_j) $ now have special representatives $ c_{1}(\mathcal{E}_j,H_j) , \dots, c_{r_j}(\mathcal{E}_j,H_j) $ given by the Chern forms of their induced Hermitian metrics. 
	
	Therefore, one can formally compute the homogeneous polynomial $ F $ given above in the $ c_{\bullet}(\mathcal{E}_j,H_j) $'s.
	The resulting closed $ (d_\rho + k, d_\rho + k) $-form on $ \FF_{\rho}(E) $ can be now pushed-forward on $ X $ through the projection $ \pi_{\rho} $ \textsl{via} integration along the fibers.
	It follows that the $(k,k)$-form
	\[
	(\pi_{\rho})_{*} F \bigl( c_{\bullet}(\mathcal{E}_{1},H_1),\dots,c_{\bullet}(\mathcal{E}_{N},H_N) \bigr)
	\]
	is a special representative for the class $ (\pi_{\rho})_{*} F \bigl( c_{\bullet}(\mathcal{E}_{1}),\dots,c_{\bullet}(\mathcal{E}_{N}) \bigr) = \Phi\bigl(c_\bullet(E)\bigr) $, where $ \Phi $ denotes the polynomial given by universal Gysin formul{\ae} in cohomology mentioned above.
	
	On the other hand, one can also formally evaluate $ \Phi $ in the Chern forms of $ (E,h) $.
	The $ (d_\rho + k, d_\rho + k) $-form $ \Phi\bigl(c_\bullet(E,h)\bigr) $ thus obtained is another special representative for the cohomology class $ \Phi\bigl(c_\bullet(E)\bigr) $.
	Consequently, the forms $ \Phi\bigl(c_\bullet(E,h)\bigr) $ and $ (\pi_{\rho})_{*} F \bigl( c_{\bullet}(\mathcal{E}_{1},H_1),\dots,c_{\bullet}(\mathcal{E}_{N},H_N) \bigr) $ differ, \textsl{a priori}, by an exact $2k$-form.
	
	The main result of this paper (see Theorem~\ref{thm: darondeau-pragacz general for differential forms}) states that this exact form is zero.
	\begin{introtheorem}\label{introthm:pushgeneral}
		We have the equality of differential forms
		\[
		(\pi_{\rho})_{*} F \bigl( c_{\bullet}(\mathcal{E}_{1},H_1),\dots,c_{\bullet}(\mathcal{E}_{N},H_N) \bigr) = \Phi\bigl(c_\bullet(E,h)\bigr) .
		\]
	\end{introtheorem}
	The major consequence of this result is that there is no longer a need to explicitly compute the integral along the fibers of $ \pi_{\rho} $ to find the resulting characteristic form on $ X $.
	Rather, it is sufficient to apply any of the Gysin formul{\ae} in cohomology mentioned above to determine the explicit expression of $ \Phi $.
	
	In the special case of the projective bundle $ \PP(E) $ of lines (resp. $ \PP(E^{\vee}) $ of hyperplanes) in $ E $, which, in our notation, corresponds to the sequence $ \rho $ of dimensions $ (0,1,r) $ (resp. $ (0,r-1,r) $), Theorem~\ref{introthm:pushgeneral} generalizes previous results of \cite{Mou04,Gul12,Div16}.
	
	More precisely, \cite[Proposition~6]{Mou04} proves, by explicit calculations, the push-forward identity in the case of a power of the first Chern form of the line bundle $ \mathcal{O}_{\PP(E)}(1) $ (resp. $ \mathcal{O}_{\PP(E^{\vee})}(1) $) equipped with the induced metric.
	In such case, the push-forward gives the (signed) \emph{Segre forms} of $ (E,h) $, which are special representatives for the Segre classes of $ E $.
	Alternative proofs of the push-forward formula for projective bundles are given in \cite[Proposition~3.1]{Gul12} (with an indirect approach) and in \cite[Proposition~1.1]{Div16} (by pointwise computations).
	
	Theorem~\ref{introthm:pushgeneral} also generalizes the recent result \cite[Main Theorem]{DF20}.
	Indeed, the latter considers only those $ F $ formally evaluated in the first Chern forms of the successive quotients of the tautological filtration over $ \FF_{\rho}(E) $, while Theorem~\ref{introthm:pushgeneral} allows to compute the push-forward of any polynomial in the Chern forms of all the possible universal vector bundles.
	See also \cite[Theorem~{A.2}]{Fin20}, which is a special case of \cite[Main Theorem]{DF20} (and, consequently, of Theorem~\ref{introthm:pushgeneral}), that provides a refinement of the Jacobi--Trudi identity on the level of differential forms.
	
	The key tool that allows us to prove Theorem~\ref{introthm:pushgeneral} is the explicit computation of the Chern curvature (at a point) of any universal vector bundle endowed with the metric induced by the given Hermitian metric $ h $ on $ E $.
	More precisely, chosen $ 0 \le \ell < l \le m $ and with respect to a natural local chart around a point of $ \FF_{\rho}(E) $, we prove that the entries of the Chern curvature of $ U_{\rho,{l}} / U_{\rho,{\ell}} $ with the induced metric are of type
	\[
	\Theta(E,h) - \textrm{vert}_{l} + \textrm{vert}_{\ell}
	\]
	where the vertical $ (1,1) $-forms $ \textrm{vert}_{l} $ and $ \textrm{vert}_{\ell} $ are such that
	\begin{center}
		$ \textrm{vert}_{l} = 0 \ $ if $ U_{\rho,l} = \pi_{\rho}^{*} E \quad $ and $ \quad \textrm{vert}_{\ell} = 0 \ $ if $ U_{\rho,{\ell}} = (0) $.
	\end{center}
	We prove these curvature formul{\ae} in Theorem~\ref{thm: curvature of natural bundles} below.
	
	Recall that \cite[Formula~(2.1)]{Mou04} provides, for $ \rho = (0,r-1,r) $, the curvature at a point of the universal quotient $ \pi_{\rho}^{*} E^{\vee} / \mathcal{O}_{\PP(E^{\vee})}(-1) $ of the projective bundle $ \PP(E^{\vee}) $, and \cite[Formula~{(4.9)}]{Dem88} establishes the curvature at a point of the successive quotients of the tautological filtration over $ \FF_{\rho}(E) $.
	
	\medskip
	The last part of the paper is devoted to two applications of Theorems~\ref{introthm:pushgeneral}.
	
	First, recall that to a given a partition $ \sigma = (\sigma_{1},\dots,\sigma_{k}) $ of a natural number $ k $ in $ k $ parts which are less or equal than $ r $, we associate the \emph{Schur polynomial} $ S_{\sigma} \in \Z[c_1,\ldots,c_r] $, of weighted degree $ 2k $, defined as $ S_{\sigma}(c_1,\ldots,c_r) := \det\left( c_{\sigma_i + j - i} \right)_{1 \le i,j \le r} $, where $ c_0 = 1 $ and $ c_s = 0 $ if $ s<0 $ or $ s>r $.
	See, for instance, \cite{Ful97} for the definition and the properties of Schur polynomials.
	
	In the first application (see Proposition~\ref{prop: schur forms as push-forwards}) we give an explicit formula which highlights how to obtain every \emph{Schur form} (\textsl{i.e.}, a Schur polynomial formally evaluated in the Chern forms) as a push-forward from the complete flag bundle, thus giving an alternative version of the Jacobi--Trudi identity for differential forms provided in \cite[Theorem~{A.2}]{Fin20}.
	
	Regarding the second application we further assume that the Hermitian holomorphic vector bundle $ (E,h) \to X $ is \emph{Griffiths semipositive} (resp. \emph{positive}).
	This means that for every $ x \in X $, $ v \in E_x $, $ \tau \in T_{X,x} $ the Chern curvature tensor of $ (E,h) $ satisfy
	\[
	h\big(\Theta(E,h)_x(\tau,\bar\tau) \cdot v,v \big) \ge 0
	\]
	(resp. $ > 0 $, and $ = 0 $ if and only if $ v $ or $ \tau $ is the zero vector).
	
	Griffiths (semi)positivity is just one of the possible positivity notions for holomorphic vector bundles of higher rank.
	We refer, for instance, to \cite{Gri69,DPS94,Laz04,Dem01,LSY13,Fin20} and the references therein for more details on the theory of positive vector bundles.
	We just recall that
	\begin{center}
		(dual) Nakano positivity $ \Rightarrow $ Griffiths positivity $ \Rightarrow $ ampleness,
	\end{center}
	and that
	\begin{center}
		(dual) Nakano semipositivity $ \Rightarrow $ Griffiths semipositivity $ \Rightarrow $ nefness.
	\end{center}
	
	Now, it is natural to expect that certain conditions of positivity on the vector bundle (as ampleness or Griffiths positivity) impose, in turn, the positivity of objects that derive from it (as characteristic classes or forms).
	A classical issue in this context, known in the literature as the {\textquotedblleft}Griffiths' conjecture on the positivity of Chern--Weil forms{\textquotedblright}, states (with a modification with respect to the original statement) the following
	\begin{conjecture}[\cite{Gri69}]
		Given any Griffiths semipositive Hermitian holomorphic vector bundle $ (E,h) \to X $ and a non-negative combination $ P $ of Schur polynomials, then the differential form $ P\big( c_{\bullet}(E,h) \big) $ obtained by formally evaluating $ P $ in the Chern forms of $ (E,h) $ is a positive differential form.
	\end{conjecture}
	Recall that a $ (k,k) $-form is \emph{positive} (following the terminology of \cite{Dem01}) if its restriction to every $ k $-dimensional complex submanifold gives a non-negative volume form.
	We refer to \cite[Section~{1.1}]{Fag20} for a summary on positivity notions for differential forms and their properties.
	For more details see \cite{HK74,Dem01}.
	
	Griffiths' conjecture can be thus interpreted as a differential, pointwise version of the celebrated Fulton--Lazarsfeld theorem \cite{FL83}, which characterizes precisely all numerically positive polynomials for ample vector bundles as the positive linear combinations of Schur polynomials, thus extending the previous works \cite{Kle69,BG71,Gie71,UT77}.
	See \cite[Theorem~2.5]{DPS94} for an extension of \cite[Theorem~3.1]{FL83} to nef vector bundles on compact K\"ahler manifolds.
	
	As a cohomology class represented by a positive differential form is numerically positive, we point out that an answer in the affirmative to Griffiths' conjecture would give a stronger (w.r.t. \cite{FL83}) statement under the stronger (but conjecturally equivalent, see \cite{Gri69}) hypothesis of Griffiths positivity.
	
	We refer to \cite{Fag20} for a detailed exposition regarding the state of the art of Griffiths' conjecture.
	Here we recall that, beside the trivial case of the first Chern form, Griffiths proved in \cite[Appendix to~\S 5.(b)]{Gri69} that the second Chern form of a rank $ 2 $ Griffiths positive vector bundle is positive and that Guler proved in \cite[Theorem~1.1]{Gul12} the positivity of the already mentioned signed Segre forms.
	More recently, it has been shown in \cite[Main Application]{DF20} the \emph{strong} positivity of the differential forms belonging to a certain sub-cone of the positive convex cone spanned by Schur forms.
	This sub-cone includes, for instance, all the signed Segre forms.
	Moreover, \cite[Theorem~{2.2}]{Fag20} points out how to prove the positivity of the second Chern form in any rank and in the semipositive context and \cite[Theorem~{0.1}]{Fag20} proves the positivity of the characteristic form $ (c_1 \wedge c_2 - c_3)(E,h) $ in rank $ 3 $.
	There are also several other related works concerning Griffiths' conjecture on positive forms, some of which are very recent, see for instance \cite{Div16,Pin18,Li20,RT21,Xia20,Fin20}.
	
	In this paper, as the second application of Theorem~\ref{introthm:pushgeneral}, and in the same spirit of \cite{Gul12,DF20}, we confirm Griffiths' conjecture for another family of characteristic forms.
	They arise by pushing-forward positive forms on the Grassmann bundles associated to the given Griffiths semipositive vector bundle, see Theorem~\ref{thm: cone of positive forms} below.
	The major difference with respect to \cite[Main~Application]{DF20} is that we need the push-forward formula for characteristic forms of flag bundles in the full generality of Theorem~\ref{introthm:pushgeneral} (indeed, Theorem~\ref{thm: cone of positive forms} does not follow from \cite[Main~Theorem]{DF20}).
	However, we can not conclude, \textsl{a priori}, that the forms obtained in Theorem~\ref{thm: cone of positive forms} are strongly positive (cf. \cite[Remark~{4.9}]{DF20}) as (squares of) $ (2,2) $-forms are involved.
	As a consequence, we also provide some explicit examples of characteristic forms whose positivity was not previously known.
	
	Finally, we want to highlight the straightforward fact that, in this context, whenever we get the positivity of a characteristic form, it gives us a pointwise inequality between products of Chern forms valid for Griffiths (semi)positive bundles over complex manifolds.
	Moreover, if the manifold is compact and of appropriate dimension, we have accordingly an inequality of Chern numbers, once we have integrated the characteristic form on the manifold (cf. \cite[Theorem~{3.2}]{Li20} and \cite[Corollary~{2.5}]{Fag20}).
	
\subsubsection*{Acknowledgments}
	Some of the results in this paper are obtained in the author's PhD thesis \cite{Fag22} defended on 12 May 2022 at Università degli Studi di Roma ``La Sapienza'' and supervised by Simone Diverio, whom the author warmly thanks for the valuable discussions and constant support.

\subsubsection*{Notation}
	For $ 0 \le p,q \le n $, $ \mathcal{A}^{p,q}(X,E) $ stands for the space $ C^{\infty}\bigl(X,\Lambda^{p,q}T_X^{\vee}\otimes E\bigr) $ of differential $ (p,q) $-forms on $ X $ with values in $ E $.
	In particular, $ \mathcal{A}^{p,q}(X) $ denotes the space $ C^{\infty}\bigl(X,\Lambda^{p,q}T_X^{\vee}\bigr) $ of differential $ (p,q) $-forms on $ X $.
	Similarly, $ \mathcal{A}^{k}(X) $ stands for the space of differential $ k $-forms on $ X $.
	
	We use the same notation $ \pi_{*} $ for both the {push-forward} of differential forms, which is given by integration along the fibers of a proper holomorphic submersion $ \pi $ of complex manifolds, and the {push-forward} induced in cohomology.

\section{Flag bundles}\label{sect: flag bundles introduction}
Let $ X $ be a complex manifold of dimension $ n $ and let $ E \to X $ be a holomorphic vector bundle of rank $ r $.
Given a natural number $ m > 0 $, fix a sequence of integers $ \rho=(\rho_0,\dots,\rho_m) $ of the form
\begin{equation*}
	0 = \rho_0 < \rho_1 < \dots < \rho_l < \dots < \rho_{m-1} < \rho_m = r .
\end{equation*}
We call such a $ \rho $ a \emph{dimension sequence}.

The \emph{flag bundle of type $ \rho $ associated to $ E $} is the holomorphic fiber bundle
\[
\pi_{\rho} \colon \FF_{\rho}(E) \to X
\]
where the fiber over $ x \in X $ is the flag manifold $ \FF_{\rho}(E_x) $, whose points $ \mathbf{f}_{x,\rho} $ are the flags of the form
\[
\{ 0_{x} \} = V_{x,\rho_0} \subset V_{x,\rho_1} \subset \dots \subset V_{x,\rho_l} \subset \dots \subset V_{x,\rho_{m-1}} \subset V_{x,\rho_m} = E_{x}
\]
where, for each $ 0 \le l \le m $, $ \dim_{\C} V_{x,\rho_l} = \rho_l $.

Over $ \FF_{\rho}(E) $ we have a tautological filtration
\begin{equation}\label{eq: tautological filtration vector bundles over flag bundle}
	(0) = U_{\rho,0} \subset U_{\rho,1} \subset \dots \subset U_{\rho,l} \subset \dots \subset U_{\rho,m-1} \subset U_{\rho,m} = \pi_{\rho}^{*}E
\end{equation}
of vector sub-bundles of $ \pi_{\rho}^{*}E $, where, for every $ 0 \le l \le m $, the fiber of $ U_{\rho,l} $ over the point $ (x,\mathbf{f}_{x,\rho}) \in \FF_{\rho}(E) $ is $ V_{x,\rho_l} $.
Therefore, the vector bundle $ U_{\rho,l} \to \FF_{\rho}(E) $ has rank $ \rho_l $.
Finally, denote by $ d_\rho $ the relative dimension of the proper holomorphic submersion $ \pi_{\rho} $.

\begin{example}
	If $ \rho $ is the complete sequence $ (0,1,\dots,r-1,r) $, then $ \FF_{\rho}(E) $ is the \emph{complete} flag bundle associated to $ E $.
	Observe that this occurs when $ m = r $.
	In this case, we shall drop the subscript $ \rho $ simply writing
	\[
	\pi \colon \FF(E) \to X
	\]
	and denoting by $ d $ the relative dimension of $ \pi $.
	Accordingly, for $ 0 \le l \le r $, the tautological filtration~\eqref{eq: tautological filtration vector bundles over flag bundle} is written as
	\begin{equation*}
		(0) = U_{0} \subset U_{1} \subset \dots \subset U_{l} \subset \dots \subset U_{m-1} \subset U_{r} = \pi^{*}E .
	\end{equation*}
\end{example}

\subsection{Universal vector bundles}
The tautological filtration~\eqref{eq: tautological filtration vector bundles over flag bundle} gives $ \binom{m+1}{2} $ \emph{universal} vector bundles over $ \FF_{\rho}(E) $, which are of the form
\[
^{\textstyle U_{\rho,{l}}}\Big/_{\textstyle U_{\rho,{\ell}}} \to \FF_{\rho}(E), \quad 0 \le \ell < l \le m .
\]
In other words, these are the tautological vector bundles already introduced (we are not considering here the vector bundle of rank $ 0 $) and all their possible quotients.

\begin{example}\label{ex: projectivized}
	Set $ m = 2 $.
	If $ \rho $ is the sequence $ (0,1,r) $ then $ \FF_{\rho}(E) $ equals the projective bundle of lines in $ E $, which we denote by $ \PP(E) $.
	In this case, the tautological filtration~\eqref{eq: tautological filtration vector bundles over flag bundle} consists on one proper sub-bundle only, namely $ U_{(0,1,r),1} $, which equals, by definition, the tautological line bundle $ \mathcal{O}_{\PP(E)}(-1) $.
	Given that $ U_{\rho,2} / U_{\rho,1} $ is the only quotient given by the filtration~\eqref{eq: tautological filtration vector bundles over flag bundle}, we have that $ \PP(E) $ has three universal vector bundles.
	They form the tautological short exact sequence
	\begin{equation}\label{eq: tautological sequence of projective bundle}
		0 \to \underbrace{\mathcal{O}_{\PP(E)}(-1)}_{U_{\rho,1}} \hookrightarrow \underbrace{\pi_{\rho}^{*} E}_{U_{\rho,2}} \twoheadrightarrow \underbrace{\pi_{\rho}^{*}E/ \mathcal{O}_{\PP(E)}(-1)}_{U_{\rho,2} / U_{\rho,1}} \to 0
	\end{equation}
	over the projective bundle $ \PP(E) $.
\end{example}

\begin{remark}\label{rem: notation for metrics}
	Suppose that $ E $ is equipped with a Hermitian metric $ h $.
	The pull-back metric $ \pi_{\rho}^{*}h $ on $ \pi_{\rho}^{*}E $ endows the vector bundle $ U_{\rho,l} $ with the restriction metric $ h_{\rho,l} := \left. \pi_{\rho}^{*} h \right|_{U_{\rho,l}} $.
	Consequently, the quotient bundle $ U_{\rho,l}/U_{\rho,\ell} $ is equipped with the Hermitian metric given by the quotient of the metrics $ h_{\rho,l} $ and $ h_{\rho,\ell} $.
	When needed, we denote such a Hermitian metric by $ h_{\rho,(l,\ell)} $.
\end{remark}

\begin{definition}
	Let $ \rho = (\rho_0,\dots,\rho_m) $ and $ \tau = (\tau_0,\dots,\tau_{\tilde{m}}) $ be two dimension sequences.
	We say that $ \rho $ is \emph{greater than or equal to} $ \tau $, and we write $ \rho \ge \tau $, if:
	\begin{enumerate}[(a)]
		\item $ m \ge \tilde{m} $;\label{conditionpt:minore}
		\item for each $ \tilde{l} \in \{ 0,\dots, \tilde{m} \} $ there is a index $ l \in \{0,\dots, m\} $ for which $ \rho_l = \tau_{\tilde{l}} $.\label{conditionpt:sequenzasuriettiva}
	\end{enumerate}
	In addition, if
	\begin{enumerate}[(c)]
		\item $ \rho \neq \tau $;
	\end{enumerate}
	we write $ \rho > \tau $.
\end{definition}

Given $ \rho > \tau $, it is useful to denote by
\[
\pi_{\tau}^{\rho} \colon \FF_{\rho}(E) \to \FF_{\tau}(E)
\]
the natural forgetful projection between flag bundles.
If $ \rho = \tau $, then $ \pi_{\rho}^{\rho} $ is simply the identity map of $ \FF_{\rho}(E) $.

By construction, $ \pi_{\tau} \circ \pi_{\tau}^{\rho} = \pi_{\rho} $, \textsl{i.e.}, the diagram
\begin{equation}\label{eq: commutative diagram of incomplete flags over the manifold}
	\begin{tikzcd}
		\FF_{\rho}(E) \arrow[dr, "\pi_{\rho}"'] \arrow[rr," \pi_{\tau}^{\rho}"]{}
		& & \FF_{\tau}(E) \arrow[dl,"\pi_{\tau}"] \\
		& X
	\end{tikzcd}
\end{equation}
is commutative.

Moreover, if $ \rho \ge \tau \ge \sigma $ then $ \pi_{\sigma}^{\tau} \circ \pi_{\tau}^{\rho} = \pi_{\sigma}^{\rho} $, \textsl{i.e.}, there is a commutative diagram
\begin{equation*}
	\begin{tikzcd}
		\FF_{\rho}(E) \arrow[dr, "\pi_{\sigma}^{\rho}"'] \arrow[rr," \pi_{\tau}^{\rho}"]{}
		& & \FF_{\tau}(E) \arrow[dl,"\pi_{\sigma}^{\tau}"] \\
		& \FF_{\sigma}(E)
	\end{tikzcd}
\end{equation*}
of projections between flag bundles over $ X $.

\begin{remark}\label{rem: universal bundles are pull-backs over projections}
	Given the projection $ \pi_{\tau}^{\rho} \colon \FF_{\rho}(E) \to \FF_{\tau}(E) $, suppose that there is a index $ \tilde{l} $ such that $ \rho_l = \tau_{\tilde{l}} $.
	Then, it is straightforward to note that the tautological vector bundle $ U_{\rho,l} \to \FF_{\rho}(E)$ is the pull-back through the map $ \pi_{\tau}^{\rho} $ of the bundle $ U_{\tau,\tilde{l}} \to \FF_{\tau}(E) $.
	Moreover, it is clear that the metric $ h_{\rho,l} $ coincides with the pull-back metric $ (\pi_{\tau}^{\rho})^{*} (h_{\tau,\tilde{l}}) $.
	Indeed,
	\begin{equation*}
		\begin{split}
			(\pi_{\tau}^{\rho})^{*} (h_{\tau,\tilde{l}}) =
			(\pi_{\tau}^{\rho})^{*} \left( { {\pi_{\tau}^{*} h}|_{U_{\tau,\tilde{l}}} } \right) =
			(\pi_{\tau}^{\rho})^{*} {\pi_{\tau}^{*} h}|_{(  \pi_{\tau}^{\rho} )^{*} U_{\tau,\tilde{l}}} =
			{ \pi_{\rho}^{*} h}|_{ U_{\rho,{l}}} =
			h_{\rho,l} .
		\end{split}
	\end{equation*}
	In addition, for $ 0 < \ell < l $, if there is a $ \tilde{\ell} $ such that $ \rho_{\ell} = \tau_{\tilde{\ell}} $, then $ U_{\rho,l}/U_{\rho,\ell} \to \FF_{\rho}(E) $ is the pull-back through $ \pi_{\tau}^{\rho} $ of the quotient $ U_{\tau,\tilde{l}}/U_{\rho,\tilde{\ell}} \to \FF_{\tau}(E) $.
	In this case, the Hermitian metric $ h_{\rho,(l,\ell)} $ is the pull-back through $ \pi_{\tau}^{\rho} $ of the metric $ h_{\tau,(\tilde{l},\tilde{\ell})} $.
\end{remark}

In particular, we emphasize the following.
\begin{remark}\label{rem: reduces to grassmannian}
	Choose $ m \ge 2 $, and let $ \rho $ be a dimension sequence.
	Fix an index $ l = 1,\dots,m-1 $ and set $ s := \rho_l $.
	Since $ \rho \ge (0, s ,r) $, we have of course the forgetful projection
	\[
	\pi_{s}^{\rho} \colon \FF_{\rho}(E) \to \GG_{s}(E)
	\]
	onto the Grassmann bundle $ \GG_{s}(E) $ of $ s $-planes in $ E $.
	
	By Remark~\ref{rem: universal bundles are pull-backs over projections}, $ U_{\rho,l} $ is the pull-back through $ \pi_{s}^{\rho} $ of the tautological $ s $-plane bundle over $ \GG_{s}(E) $ denoted by $ \gamma_{s} $.
	Moreover, the Hermitian metric $ h_{\rho,l} $ is the pull-back of the obvious restriction metric $ h_s $ on $ \gamma_{s} $ induced by $ h $.
	Similarly, $ (\pi_{\rho}^{*}E / U_{\rho,l},h_{\rho,(m,l)}) $ is the pull-back through $ \pi_{s}^{\rho} $ of the universal quotient bundle over $ \GG_{s}(E) $, equipped with the quotient metric.
	
	Finally, if $ m > 2 $ and $ 0 < \ell < l < m $, let $ \mathbf{s} $ be the sequence $ (0,s_1,s_2,r) $ with $ s_1 = \rho_{\ell} $, $ s_2 = \rho_l $.
	Given that $ \rho \ge \mathbf{s} $, we have that $ U_{\rho,l}/U_{\rho,\ell} $ is the pull-back through the map $ \pi_{\mathbf{s}}^{\rho} $ of the unique quotient $ U_{\mathbf{s},2} / U_{\mathbf{s},1} $ of proper sub-bundles over the flag bundle $ \FF_{\mathbf{s}}(E) $.
	As always, this vector bundle is equipped with the obvious quotient metric $ h_{\mathbf{s},(2,1)} $ introduced above, and $ (\pi_{\mathbf{s}}^{\rho})^{*} h_{\mathbf{s},(2,1)} = h_{\rho,(l,\ell)} $.
\end{remark}

The observations in Remark~\ref{rem: reduces to grassmannian} are the starting point for the proof of Theorem~\ref{thm: curvature of natural bundles}.

\subsection{Local frames of universal vector bundles}\label{sect: coordinates flag and frames universal bundles}
Fix a point $ x_0 \in X $ and local holomorphic coordinates $ z = (z_1,\dots,z_n) $ on an open set of $ X $ centered at $ x_0 $.
Recall that a local holomorphic frame $ (e_1,\dots,e_r) $ of $ E $ centered at $ x_0 $ is called a \emph{normal coordinate frame} at $ x_0 $ (see,~\cite[(12.10)~Proposition]{Dem01}) if
\begin{equation*}
	\langle e_{\alpha}(z), e_{\beta}(z) \rangle_h = \delta_{\alpha\beta} - \sum_{1 \le j,k \le n} c_{jk\alpha\beta} \, z_j \bar{z}_k + \operatorname{O}(|z|^3) ,
\end{equation*}
where the $ c_{jk\alpha\beta} $'s are the coefficients of the Chern curvature tensor
\begin{equation*}
	\Theta(E,h)_{x_0} = \sum_{1 \le \alpha,\beta \le r} \sum_{1 \le j,k \le n} c_{jk\alpha\beta} \, dz_j \wedge d\bar{z}_k \otimes e_{\alpha}^{\vee} \otimes e_{\beta}
\end{equation*}
expressed with respect to the frame $ (e_1,\dots,e_r) $.

Since $ i \Theta(E,h) $ is a $ (1,1) $-form with values in the bundle $ \operatorname{Herm}(E,h) $ of Hermitian endomorphisms of $ E $, we have the symmetry relation $ \bar{c}_{jk\alpha\beta} = c_{kj\beta\alpha} $.

In this notation, the $ (\beta,\alpha) $-entry of the matrix associated to $ \Theta(E,h)_{x_0} $ with respect to $ (e_1,\dots,e_r) $ in the coordinate open set considered is the $ (1,1) $-form
\begin{equation*}
	\Theta_{\beta\alpha} := \sum_{1 \le j,k \le n} c_{jk\alpha\beta} \, dz_j \wedge d\bar{z}_k.
\end{equation*}

Now, given a dimension sequence $ \rho $, fix a flag $ \mathbf{f}_0 \in \FF_{\rho}(E_{x_0}) $.
Let $ (e_1,\dots,e_r) $ be a local normal frame of $ E $ at $ x_0 $.
As the action of the general linear group is transitive on $ \FF_{\rho}(E_{x_0}) $, we assume that the flag
\begin{equation*}
	\{ 0_{x_0} \} \subset \dots \subset \operatorname{Span}\bigl\{ e_{r-\rho_l+1}(x_0),e_{r-\rho_l+2}(x_0),\dots,e_{r}(x_0) \bigr\} \subset \dots \subset E_{x_0}
\end{equation*}
coincides with $ \mathbf{f}_0 $.

The basis $ \bigl( e_1(z), \dots, e_r(z) \bigr)$ gives affine coordinates $ \zeta = (\zeta_{\lambda\mu}) $ on the fiber $ \FF_{\rho}(E_z) $ where the indices $ \lambda $ and $ \mu $ satisfy the condition
\begin{equation*}\label{eq:asterisk}
	\text{ there exists } \ell = 1, \dots, m-1 \text{ for which } 1 \le \lambda \le r - \rho_{m-\ell} < \mu \le r . \tag{$*$}
\end{equation*}
Such coordinates parameterize flags of $ E_{z} $ of type
\[
\{ 0_{z} \} \subset \dots \subset V_{z,\rho_{l}} := \operatorname{Span} \bigl\{\epsilon_{r-\rho_l+1}(z,\zeta),\epsilon_{r-\rho_l+2}(z,\zeta),\dots,\epsilon_{r}(z,\zeta) \bigr\} \subset \dots \subset E_{z}
\]
where, for $ 1 \le \alpha \le r $,
\begin{equation*}
	\epsilon_{\alpha}(z,\zeta)=e_{\alpha}(z) + \sum \zeta_{\lambda \alpha}\, e_\lambda(z)
\end{equation*}
and the summation is taken over all $ 1 \le \lambda < \alpha $ such that, as before, the indices $ \lambda $ and $ \alpha $ satisfy Condition~\eqref{eq:asterisk}.
Summing up, we have constructed in this way local holomorphic coordinates $ (z,\zeta) = \left( z_1,\dots,z_n, \zeta_{\lambda\mu} \right) $ on $ \FF_{\rho}(E) $ centered at $ (x_0,\mathbf{f}_0) $.

By construction, if $ 0 < l < m $, the local sections
\begin{equation}\label{eq: frame of subbundles}
	\epsilon_\alpha (z,\zeta) , \quad r-\rho_{l} < \alpha \le r ,
\end{equation}
form a local holomorphic frame of $ U_{\rho,l} $.
Moreover, a local holomorphic frame for $ \pi_{\rho}^{*}E/U_{\rho,l} $ is given by the sections
\begin{equation}\label{eq: frame of quotients}
	\tilde{e}_\alpha (z,\zeta) = \text{image of $ e_\alpha(z) $ in $ ^{\textstyle E_{z}}\Big/_{\textstyle V_{z,\rho_l}} $} , \quad 1 \le \alpha \le r-\rho_{l} .
\end{equation}
Similarly, if $ 0 < \ell < l < m $, the sections
\begin{equation}\label{eq: frame of proper quotients}
	\tilde{\epsilon}_\alpha (z,\zeta) = \text{image of $ \epsilon_\alpha(z,\zeta) $ in $ ^{\textstyle V_{z,\rho_l}}\Big/_{\textstyle V_{z,\rho_{\ell}}} $} , \quad r-\rho_{l} < \alpha \le r-\rho_{\ell} ,
\end{equation}
form a local holomorphic frame for $ U_{\rho,l}/U_{\rho,\ell} $.

In Section~\ref{sect: curvature universal bundles} we use the local frames~\eqref{eq: frame of subbundles},~\eqref{eq: frame of quotients},~\eqref{eq: frame of proper quotients} to compute the Chern curvature tensors at a point of all the $ \binom{m+1}{2} $ universal vector bundles of $ \FF_{\rho}(E) $.

\section{Curvature of universal vector bundles}\label{sect: curvature universal bundles}
Let $ \rho = (\rho_{0},\dots,\rho_{m}) $ be a dimension sequence and $ \pi_{\rho} \colon \FF_{\rho}(E) \to X $ the associated flag bundle.
In the notation of Section~\ref{sect: flag bundles introduction}, we now compute the Chern curvature tensor of all the universal vector bundles of $ \FF_{\rho}(E) $ endowed with the natural Hermitian metrics induced (see Remark~\ref{rem: notation for metrics}).

\begin{theorem}\label{thm: curvature of natural bundles}
	Choose a point $ (x_0,\mathbf{f}_0) \in \FF_{\rho}(E) $ and let $ (e_1,\dots,e_r) $ be a local normal frame of $ (E,h) $ at $ x_0 $ which identifies the flag $ \mathbf{f}_0 $.
	Fix the local holomorphic coordinates $ (z,\zeta) $ centered at $ (x_0,\mathbf{f}_0) $ introduced in Section~\ref{sect: coordinates flag and frames universal bundles}.
	
	Then, for any $ 0 \le \ell < l \le m $, the curvature $ \Theta(U_{\rho,l}/U_{\rho,\ell},h_{\rho,(l,\ell)})_{(x_0,\mathbf{f}_0)} $ equals
		\begin{equation*}
			\sum_{r-\rho_l < \alpha,\beta \le r-\rho_{\ell}}
			\left(
			\Theta_{\beta\alpha}
			- \sum_{1 \le \lambda \le r-\rho_l} d\zeta_{\lambda\alpha} \wedge d\bar{\zeta}_{\lambda\beta}
			+ \sum_{r-\rho_{\ell} < \mu \le r} d{\zeta}_{\beta\mu} \wedge d\bar{\zeta}_{\alpha\mu}
			\right) \otimes {e_\alpha^\vee} \otimes {e_\beta} .
		\end{equation*}

\end{theorem}

By a slight abuse of notation, we have identified the $ (1,1) $-form $ \Theta_{\beta\alpha} $ with its pull-back $ \pi_{\rho}^{*} \Theta_{\beta\alpha} $ along $ \pi_{\rho} $, and we have omitted the dependence on $ x_0 $ in the endomorphism part $ {e_\alpha^\vee} \otimes {e_\beta} $.

\begin{proof}[Proof of Theorem~\ref{thm: curvature of natural bundles}, 1\textsuperscript{st} part]
	First, we want to prove Theorem~\ref{thm: curvature of natural bundles} when $ 0 = \ell < l \le m $.
	This corresponds to the case of tautological vector bundles.
	
	Set $ s := \rho_l $.
	By Remark~\ref{rem: reduces to grassmannian} the curvature $ \Theta(U_{\rho,l}, h_{\rho,l}) $ is just the pull-back through the projection $ \pi_{(0,s,r)}^{\rho} $ of the curvature $ \Theta(\gamma_{s}, h_s) $ of the tautological vector bundle $ (\gamma_{s}, h_s) \to \GG_{s}(E) $.
	Thus, we can assume that $ \FF_{\rho}(E) = \GG_{s}(E) $ and $ U_{\rho,l} = \gamma_{s} $.
	
	A local holomorphic frame for $ \gamma_{s} $ is given by the sections
	\begin{equation*}
		\epsilon_\alpha (z,\zeta) = e_\alpha(z) + \sum_{1 \le \lambda \le r-s} \zeta_{\lambda\alpha} e_\lambda(z) , \quad \text{for} \ r-s < \alpha \le r .
	\end{equation*}
	Observe that, for the sequence $ (0,s,r) $, Condition~\eqref{eq:asterisk} on $ \lambda $ and $ \alpha $ for the coordinate $ \zeta_{\lambda\alpha} $ becomes $ 1 \le \lambda \le r-s < \alpha \le r $.
	We want to show that $ (\epsilon_{r-s+1},\dots,\epsilon_r) $ is also a local normal frame of $ (\gamma_{s},h_s) $ at $ (x_0,\mathbf{f}_0) $.
	
	To simplify the notation in what follows, we omit the dependence on the coordinates.
	The $ (\alpha,\beta) $-entry of the Hermitian matrix associated to $ h_s $ with respect to the coordinates $ y := (z,\zeta) $ is
	\begin{equation}\label{eq: metric of tautological frame}
		\begin{split}
			&\langle \epsilon_\alpha, \epsilon_\beta \rangle \\
			&= \left< e_\alpha + \sum_{\lambda} \zeta_{\lambda\alpha} e_\lambda, e_\beta + \sum_{\mu} \zeta_{\mu\beta} e_\mu  \right> \\
			&= \left< e_\alpha, e_\beta \right> +
			\sum_{\mu} \bar{\zeta}_{\mu\beta} \left< e_\alpha, e_\mu  \right> +
			\sum_{\lambda} \zeta_{\lambda\alpha} \left< e_\lambda, e_\beta \right> +
			\sum_{\lambda,\mu} \zeta_{\lambda\alpha} \bar{\zeta}_{\mu\beta} \left< e_\lambda, e_\mu  \right> \\
			&= \delta_{\alpha\beta} - \sum_{j,k} c_{jk\alpha\beta} z_j \bar{z}_k + \sum_{\mu} \bar{\zeta}_{\mu\beta} \delta_{\alpha\mu} + \sum_{\lambda} \zeta_{\lambda\alpha} \delta_{\lambda\beta} + \sum_{\lambda,\mu} \zeta_{\lambda\alpha}\bar{\zeta}_{\mu\beta} \delta_{\lambda\mu} + \operatorname{O}(|y|^3) .
		\end{split}
	\end{equation}
	Now, observe that $ \sum_{\mu} \bar{\zeta}_{\mu\beta} \delta_{\alpha\mu} = 0 $ since $ \mu \le r-s < \alpha $.
	The same holds for $ \sum_{\lambda} \zeta_{\lambda\alpha} \delta_{\lambda\beta} $ as $ \lambda \le r-s < \beta $.
	We have shown that
	\begin{equation*}
		\langle \epsilon_\alpha, \epsilon_\beta \rangle = \delta_{\alpha\beta} - \sum_{1 \le j,k \le n} c_{jk\alpha\beta} z_j \bar{z}_k + \sum_{1 \le \lambda \le r-s} \zeta_{\lambda\alpha}\bar{\zeta}_{\lambda\beta} + \operatorname{O}(|y|^3)
	\end{equation*}
	thus, $ (\epsilon_{r-s+1},\dots,\epsilon_r) $ is a normal coordinate frame at $ (x_0,\mathbf{f}_0) $.
	
	Therefore, the $ (\beta,\alpha) $-entry of the curvature matrix $ \Theta(\gamma_{s},h_s)_{(x_0,\mathbf{f}_0)} $ with respect to the coordinates $ y $ is
	\begin{equation*}
		\sum_{1 \le j,k \le n} c_{jk\alpha\beta} dz_j \wedge d\bar{z}_k - \sum_{1 \le \lambda \le r-s} d\zeta_{\lambda\alpha} \wedge d\bar{\zeta}_{\lambda\beta}
	\end{equation*}
	and the first part of the proof is over.
\end{proof}

\begin{proof}[Proof of Theorem~\ref{thm: curvature of natural bundles}, 2\textsuperscript{nd} part]
	Now, suppose that $ 0 < \ell < l \le m $.
	We are going to compute the Chern curvature of all the possible universal quotients.
	
	Define the sequence $ \mathbf{s} $ as $ (0,s_1=\rho_{\ell},s_2=\rho_{l},r) $.
	Again by Remark~\ref{rem: reduces to grassmannian} we note that it is sufficient to compute the curvature of $ U_{\mathbf{s},2} / U_{\mathbf{s},1} $ (equipped with the quotient of the restriction metrics) at a point of $ \FF_{\mathbf{s}}(E) $.
	Observe that there is a slight abuse of notation when $ 0 < \ell < l = m $.
	In this case, we shall consider the universal quotient of the Grassmann bundle $ \GG_{s_1}(E) $ and the following calculations also cover this case.
	
	The local holomorphic coordinates $ (z,\zeta) $ on $ \FF_{\mathbf{s}}(E) $ centered at $ (x_0,\mathbf{f}_0) $ parameterize the flags $ 0 \subset V_{z,s_1} \subset V_{z,s_2} \subseteq E_{z} $.
	Thus, the local holomorphic frame \eqref{eq: frame of proper quotients} for $ U_{\mathbf{s},2} / U_{\mathbf{s},1} $ is given by the sections
	\begin{equation*}
		\tilde{\epsilon}_\alpha (z,\zeta) = \text{image of $ \epsilon_\alpha(z,\zeta) $ in $ V_{z,s_2} / V_{z,s_1} $} , \quad \text{for} \ r-s_2 < \alpha \le r-s_1 .
	\end{equation*}
	Consider the tautological exact sequence
	\[
	0 \to U_{\mathbf{s},1} \overset{\iota}{\hookrightarrow} U_{\mathbf{s},2} \overset{p}{\twoheadrightarrow} U_{\mathbf{s},2} / U_{\mathbf{s},1} \to 0
	\]
	over $ \FF_{\mathbf{s}}(E) $, and denote by $ p^{\star} \colon U_{\mathbf{s},2} / U_{\mathbf{s},1} \to U_{\mathbf{s},2} $ be the $ C^{\infty} $ orthogonal splitting of the natural projection $ p $.
	We have,
	\begin{equation*}
		p^{\star} \cdot \tilde{\epsilon}_\alpha (z,\zeta) = \epsilon_\alpha(z,\zeta) + \sum_{r-s_1 < \mu \le r} u_{\alpha\mu}(z,\zeta) \epsilon_\mu(z,\zeta) ,
	\end{equation*}
	for some smooth function $ u_{\alpha\mu} $.
	As before, to simplify the notation set $ y := (z,\zeta) $.
	For $ r- s_1 < \beta \le r $, by using the computations made in~\eqref{eq: metric of tautological frame}, we have
	\begin{equation*}
		\begin{split}
			0 &= \langle \tilde{\epsilon}_\alpha, p \cdot \epsilon_\beta \rangle \\
			&= \langle p^{\star} \cdot \tilde{\epsilon}_\alpha, \epsilon_\beta \rangle \\
			&= \left< \epsilon_\alpha + \sum_{r-s_1 < \mu} u_{\alpha\mu} \epsilon_\mu, \epsilon_\beta  \right> \\
			&= \left< \epsilon_\alpha, \epsilon_\beta \right> +
			\sum_{r-s_1 < \mu} u_{\alpha\mu} \left< \epsilon_\mu, \epsilon_\beta \right> \\
			&= \delta_{\alpha\beta} - \sum_{j,k} c_{jk\alpha\beta} z_j \bar{z}_k + \bar{\zeta}_{\alpha\beta} + \sum_{\lambda \le r-s_2} \zeta_{\lambda\alpha} \bar{\zeta}_{\lambda\beta} \\
			&+ u_{\alpha\beta} + \sum_{r-s_1 < \mu} u_{\alpha\mu}
			\left( - \sum_{j,k} c_{jk\mu\beta} z_j \bar{z}_k + \sum_{\nu \le r-s_1} \zeta_{\nu\mu} \bar{\zeta}_{\nu\beta} \right) + \operatorname{O}(|y|^3) .
		\end{split}
	\end{equation*}
	
	Note that the terms $ \sum_{\nu \le r - s_1} \bar{\zeta}_{\nu\beta} \delta_{\alpha\nu} $ and  $ \sum_{\sigma \le r - s_2} {\zeta}_{\sigma\alpha} \delta_{\sigma\beta} $ coming from the product $ \left< \epsilon_\alpha, \epsilon_\beta \right> $ equals $ \bar{\zeta}_{\alpha\beta} $ and 0 respectively.
	This is because of the inequality $ r - s_2 < \alpha \le r - s_1 < \beta $.
	Instead, the terms $ \sum_{\nu \le r - s_1} \bar{\zeta}_{\nu\beta} \delta_{\mu\nu} $ and  $ \sum_{\sigma \le r - s_2} {\zeta}_{\sigma\mu} \delta_{\sigma\beta} $ coming from the product $ \left< \epsilon_\mu, \epsilon_\beta \right> $ are both zero, as $ r-s_1 < \mu, \beta $.
	
	Therefore, we have obtained the equality $ u_{\alpha\beta} = - \bar{\zeta}_{\alpha\beta} + \operatorname{O}(|y|^2) $ and, consequently,
	\[
	p^{\star} \tilde{\epsilon}_\alpha = \epsilon_\alpha - \sum_{r-s_1 < \mu} \left( \bar{\zeta}_{\alpha\mu} + \operatorname{O}(|y|^2) \right) \epsilon_\mu .
	\]
	
	Denoting by $ b \in \mathcal{A}^{1,0}\big( \FF_{\mathbf{s}}(E), \operatorname{Hom}(U_{\mathbf{s},1},U_{\mathbf{s},2} / U_{\mathbf{s},1}) \big) $ the second fundamental form of $ U_{\mathbf{s},1} $ in $ U_{\mathbf{s},2} $ and by $ b^{\star} \in \mathcal{A}^{0,1}\big( \FF_{\mathbf{s}}(E), \operatorname{Hom}(U_{\mathbf{s},2} / U_{\mathbf{s},1},U_{\mathbf{s},1}) \big) $ its adjoint, we have
	\begin{equation*}
		- \iota \circ b^{\star} \cdot \tilde{\epsilon}_\alpha \rvert_{(x_0,\mathbf{f}_0)} = \bar{\partial} p^{\star} \tilde{\epsilon}_\alpha \rvert_{(0,0)} = - \sum_{r-s_1 < \mu} d \bar{\zeta}_{\alpha\mu} \otimes \epsilon_\mu .
	\end{equation*}
	Therefore, the curvature $ \Theta(U_{\mathbf{s},2} / U_{\mathbf{s},1},h_{\mathbf{s},(2,1)})_{(x_0,\mathbf{f}_0)} $ with respect to the coordinates $ y $ is
	\begin{equation*}
		\begin{split}
			& \bigg( \Theta(U_{\mathbf{s},2},h_{\mathbf{s},2}) \big|_{U_{\mathbf{s},2} / U_{\mathbf{s},1}} + b \wedge b^{\star} \bigg)_{(x_0,\mathbf{f}_0)}\\
			&= \sum_{\alpha,\beta}
			\left( \Theta_{\beta\alpha} - \sum_{1 \le \lambda \le r-s_2} d\zeta_{\lambda\alpha} \wedge d\bar{\zeta}_{\lambda\beta} + \sum_{r-s_1 < \mu \le r} d{\zeta}_{\beta\mu} \wedge d\bar{\zeta}_{\alpha\mu} \right) \otimes \tilde{\epsilon}_\alpha^{\vee} \otimes \tilde{\epsilon}_\beta
		\end{split}
	\end{equation*}
	where the equality follows from the expressions
	\begin{equation*}
		b_{(x_0,\mathbf{f}_0)}^{\star} = \sum_{\substack{r-s_2 < \alpha \le r-s_1 \\ r-s_1 < \mu}} d \bar{\zeta}_{\alpha\mu} \otimes \tilde{\epsilon}_\alpha^{\vee} \otimes \epsilon_\mu,
		\quad
		b_{(x_0,\mathbf{f}_0)} = \sum_{\substack{r-s_2 < \beta \le r-s_1 \\ r-s_1 < \lambda}} d {\zeta}_{\beta\lambda} \otimes \epsilon_\lambda^{\vee} \otimes \tilde{\epsilon}_\beta
	\end{equation*}
	and from the computation of $ \Theta(U_{\mathbf{s},2},h_{\mathbf{s},2}) $ made in the 1\textsuperscript{st} part of the proof.
\end{proof}

The second part of the proof above follow the computations for the curvature of the tautological and universal quotient bundles of the Grassmann manifolds given in \cite[\S~16.C]{Dem01}.

\section{Pointwise universal Gysin formul{\ae} for flag bundles}
From now on, fix two indices $ \ell $ and $ l $ such that $ 0 \le \ell < l \le m $.
In order to simplify the notation, we denote by $ \mathcal{E} $ the universal vector bundle $ U_{\rho,l}/U_{\rho,\ell} $ of $ \FF_{\rho}(E) $.
Finally, let $ H $ be the Hermitian metric $ h_{\rho,(l,\ell)} $ naturally induced by $ h $ on $ \mathcal{E} $, as mentioned in Remark~\ref{rem: notation for metrics}.

\subsection{Intrinsic expression of the curvature}\label{sect: intrinsic curvature vector bundle}
In this section we express intrinsically the Chern curvature of the universal vector bundles.
To do this we use the formul{\ae} provided in our Theorem~\ref{thm: curvature of natural bundles}.

Recall that we have a natural (smooth) orthogonal splitting of the tangent bundle $ T_{\FF_{\rho}(E)} $ as the direct sum of the relative tangent bundle $ T_{\FF_{\rho}(E) / X} $ plus a horizontal part $ T_{\FF_{\rho}(E) / X}^{\perp_h} $, which depends only on $ h $ (see, for instance, \cite{DF20}).
Therefore, we can write the curvature tensor $ \Theta(\mathcal{E},H) $ as the sum of a {\textquotedblleft}{vertical}{\textquotedblright} tensor
\begin{equation*}
	\Theta_{(\mathcal{E},H)}^{\operatorname{vert}} \in C^{\infty} \left( \FF_{\rho}(E), \Lambda^{1,1}T_{\FF_{\rho}(E) / X}^{\vee} \otimes \operatorname{End}(\mathcal{E}) \right)
\end{equation*}
plus a {\textquotedblleft}{horizontal}{\textquotedblright} tensor
\begin{equation*}
	\Theta_{(\mathcal{E},H)}^{\operatorname{hor}} \in C^{\infty} \left( \FF_{\rho}(E), \Lambda^{1,1}(T_{\FF_{\rho}(E) / X}^{\perp_h})^{\vee} \otimes \operatorname{End}(\mathcal{E}) \right) .
\end{equation*}

Our aim is to give an explicit expression of $ \Theta_{(\mathcal{E},H)}^{\operatorname{hor}} $, generalizing what was done in \cite[Section~{2.1}]{DF20} for the curvature of universal line bundles.

If $ p_2 \colon T_{\FF_{\rho}(E)} \to T_{\FF_{\rho}(E)/X}^{\perp_h} $ is the projection on $ T_{\FF_{\rho}(E) / X}^{\perp_h} $ relative to the above mentioned splitting, we have that
\begin{equation}\label{eq: horizontal curvature compose projection}
	\Theta_{(\mathcal{E},H)}^{\operatorname{hor}} = \Theta(\mathcal{E},H) \circ (p_2 \otimes \bar{p}_2)
\end{equation}
where the right hand side of Formula~\eqref{eq: horizontal curvature compose projection} is a slightly improper notation which means that we are composing the $ (1,1) $-form part of $ \Theta(\mathcal{E},H) $ with $ p_2 \otimes \bar{p}_2 $.

Fix a point $ (x_0,\mathbf{f}_0) \in \FF_{\rho}(E) $.
With respect to the coordinates $ (z,\zeta) $ introduced in Section~\ref{sect: coordinates flag and frames universal bundles}, Theorem~\ref{thm: curvature of natural bundles} and Formula~\eqref{eq: horizontal curvature compose projection} give the equality
\begin{equation}\label{eq: curvature tensor of horizontal part}
	\Theta_{(\mathcal{E},H)}^{\operatorname{hor}}(x_0,\mathbf{f}_0) =
	\sum_{r-\rho_l < \alpha,\beta \le r-\rho_{\ell}} \Theta_{\beta \alpha} \otimes e_\alpha^{\vee}(x_0) \otimes e_\beta(x_0) .
\end{equation}
Observe that, whichever the local frame of $ \mathcal{E} $, that is~\eqref{eq: frame of subbundles},~\eqref{eq: frame of quotients} or~\eqref{eq: frame of proper quotients}, we can suppose that it coincides with $ \big( e_{r-\rho_l+1}(x_0),\dots,e_{r-\rho_\ell}(x_0) \big) $ if evaluated in $ (x_0,\mathbf{f}_0) $, since $ (e_1,\dots,e_r) $ is a local normal frame at $ x_0 $.

Now, we define a section
$$ \theta_{(\ell,l)} \colon \FF_{\rho}(E) \to {\Lambda}^{1,1}{T_{\FF_{\rho}(E)}^{\vee}} \otimes \operatorname{End(\mathcal{E})} $$
as follows.
For $ x \in X, $ let $ \mathbf{f} \in \FF_{\rho}(E_x) $ be given by a unitary basis $ (v_1,\dots,v_r) $ of $ E_x $.
We set
\begin{equation}\label{eq: horizontal curvature explicit}
	\theta_{(\ell,l)}(x,\mathbf{f}) = \frac{i}{2\pi} \sum_{r-\rho_l < \lambda,\mu \le r-\rho_{\ell}} \bigl\langle \pi_{\rho}^*\Theta(E,h)_{(x,\mathbf{f})}\cdot {v_\lambda}, {v_\mu} \bigr\rangle_h \otimes v_{\lambda}^{\vee} \otimes v_{\mu} . 
\end{equation}

\begin{lemma}\label{lem:blocksgeneral}
	The section $ \theta_{(\ell,l)} $ is well defined, \textsl{i.e.} it does not depend upon the choice of a particular representative $ \mathbf v= (v_1,\dots,v_r) $ for $ \mathbf{f} $.
\end{lemma}
\begin{proof}
	Take a local normal frame $ (e_1,\dots,e_r) $ at $ x $ such that $ \mathbf{e} := \bigl(e_1(x),\dots,e_r(x)\bigr) $ and $ \mathbf{v} $ identify the same flag $ \mathbf{f} $.
	Consequently, we know that if an index $ \lambda $ is such that
	\[
	r-\rho_l < \lambda \le r-\rho_{\ell}
	\]
	then we can write
	\begin{equation*}
		v_\lambda = \sum_{r-\rho_l < \alpha \le r-\rho_{\ell}} a_{\alpha \lambda} e_{\alpha}(x),
	\end{equation*}
	where the coefficients $ (a_{pq}) $ are the entries of the change of coordinates matrix
	\begin{equation*}
		\begin{pmatrix}
			A_{11} & 0 & 0 \\
			0 & \ddots & 0 \\
			0 & 0 & A_{mm}
		\end{pmatrix}.
	\end{equation*}
	Notice that the diagonal block $ A_{jj} $ is again a unitary matrix of size $ \rho_{m-j+1} - \rho_{m-j} $.
	
	To simplify the notation in what follows, we omit the dependence on $ x $ (resp. $ (x,\mathbf{f}) $) of $\mathbf{e}$ (resp. $  \pi_{\rho}^*\Theta(E,h)_{(x,\mathbf{f})} $).
	We thus get the chain of equalities
	\begin{align*}
		& \frac{i}{2\pi} \sum_{\lambda,\mu} \bigl\langle \pi_{\rho}^*\Theta(E,h)\cdot {v_\lambda}, {v_\mu} \bigr\rangle_h \otimes v_{\lambda}^{\vee} \otimes v_{\mu} \\
		&= \frac{i}{2\pi} \sum_{\lambda,\mu} {\left\langle {\pi_{\rho}^*\Theta(E,h)}\cdot \left( \sum_{\alpha} a_{\alpha \lambda} e_\alpha \right), \sum_{\beta} a_{\beta \mu} e_\beta \right\rangle_h} \otimes \left( \sum_{\tilde{\alpha}} a_{\tilde{\alpha} \lambda} e_{\tilde{\alpha}} \right)^{\vee} \otimes \sum_{\tilde{\beta}} a_{\tilde{\beta} \mu} e_{\tilde{\beta}} \\
		&= \frac{i}{2\pi} \sum_{\alpha,\tilde{\alpha}}\underbrace{ \sum_{\lambda} a_{\alpha \lambda} \overline{a_{\tilde{\alpha} \lambda}}}_{=\delta_{\alpha\tilde{\alpha}}} \sum_{\tilde{\beta},\beta}\underbrace{ \sum_{\mu} a_{\tilde{\beta} \mu} \overline{a_{\beta \mu}}}_{=\delta_{\tilde{\beta}\beta}} \,\bigl\langle {\pi_{\rho}^*\Theta(E,h)}\cdot {e_\alpha}, {e_\beta} \bigr\rangle_h \otimes e_{\tilde{\alpha}}^{\vee} \otimes e_{\tilde{\beta}} \\
		&= \frac{i}{2\pi} \sum_{\alpha,\beta}\bigl\langle {\pi_{\rho}^*\Theta(E,h)}\cdot {e_\alpha}, {e_\beta} \bigr\rangle_h \otimes e_{\alpha}^{\vee} \otimes e_{\beta},
	\end{align*}
	where the indices of the summations are such that
	\[
	\lambda,\mu,\alpha,\beta,\tilde{\alpha},\tilde{\beta} = r-\rho_l + 1, \dots, r-\rho_{\ell} .
	\]
	Hence the lemma follows.
\end{proof}
Observe also that, by definition, the section $ \theta_{(\ell,l)} $ is smooth, hence it belongs to the space $ \mathcal{A}^{1,1}\bigl(\FF_{\rho}(E),\operatorname{End}(\mathcal{E})\bigr) $.

Now we show that $ \theta_{(\ell,l)} $ coincides in each point of the flag bundle with the $ h $-horizontal part of the Chern curvature of $ (\mathcal{E},H) $.
\begin{lemma}\label{lem: definizione di theta generale}
	The equality
	\[
	\theta_{(\ell,l)} = \frac{i}{2\pi}\Theta_{(\mathcal{E},H)}^{\operatorname{hor}}
	\]
	holds.
\end{lemma}
\begin{proof}
	At any given $ (x,\mathbf{f}) \in \FF_{\rho}(E) $, choose $ (e_1,\dots,e_r) $ to be a local normal frame for $ E $ at $ x $ such that $ \mathbf{f} $ is given by $ \bigl(e_1(x),\dots,e_r(x)\bigr) $, and consider the induced holomorphic coordinates around $ (x,\mathbf{f})$ as in Section~\ref{sect: coordinates flag and frames universal bundles}.
	
	By Lemma~\ref{lem:blocksgeneral} we have the following chain of equalities
	\begin{align*}
		\theta_{(\ell,l)}(x,\mathbf{f}) &= \frac{i}{2\pi} \sum_{r-\rho_l < \lambda,\mu \le r-\rho_{\ell}} \bigl\langle \pi_{\rho}^*\Theta(E,h)_{(x,\mathbf{f})}\cdot {e_\lambda(x)}, {e_\mu(x)} \bigr\rangle_h \otimes {e_{\lambda}(x)}^{\vee} \otimes e_{\mu}(x) \\
		&= \frac{i}{2\pi} \sum_{r-\rho_l < \lambda,\mu \le r-\rho_{\ell}} \pi_{\rho}^*\Theta_{\mu\lambda}(x,\mathbf{f}) \otimes {e_{\lambda}(x)}^{\vee} \otimes e_{\mu}(x) \\
		&= \frac{i}{2\pi} \sum_{r-\rho_l < \lambda,\mu \le r-\rho_{\ell}} \Theta_{\mu\lambda}(x) \otimes {e_{\lambda}(x)}^{\vee} \otimes e_{\mu}(x) \\
		&= \frac{i}{2\pi} \Theta_{(\mathcal{E},H)}^{\operatorname{hor}}(x,\mathbf{f}),
	\end{align*}
	where the last equality follows from Formula~\eqref{eq: curvature tensor of horizontal part}.
\end{proof}

Now we relabel the vertical part $ \frac{i}{2\pi} \Theta_{(\mathcal{E},H)}^{\operatorname{vert}} $ as $ \omega_{(\ell,l)} $.
Obviously, $ \omega_{(\ell,l)} $ is smooth as it equals the difference $ \frac{i}{2\pi} \Theta(\mathcal{E},H) - \theta_{(\ell,l)} $.

Therefore, by Lemma~\ref{lem:blocksgeneral} and Lemma~\ref{lem: definizione di theta generale} we have proved the following.
\begin{proposition}
	For $ 0 \le \ell < l \le m $, we have the equality
	\begin{equation}\label{eq: curvature of universal is sum of two tensors}
		\frac{i}{2\pi} \Theta(\mathcal{E},H) = \theta_{(\ell,l)} + \omega_{(\ell,l)} .
	\end{equation}
\end{proposition}

\subsection{Main result}
Set $ N := \binom{m+1}{2} $, and let $ (\mathcal{E}_1,H_1), \dots, (\mathcal{E}_N,H_N) $ be an enumeration of all the universal vector bundles, of ranks $ r_1, \dots, r_N $ respectively, over $ \FF_{\rho}(E) $.
Let $ F $ be a complex weighted homogeneous polynomial of degree $ 2(d_\rho + k) $ in $ r_1 + \dots + r_N $ variables.

We state the following technical proposition.
\begin{proposition}\label{lem: pol homogeneous general}
	Let $ F\bigl(c_{\bullet}(\mathcal{E}_1,H_1),\dots,c_{\bullet}(\mathcal{E}_N,H_N)\bigr) $ be a complex homogeneous polynomial in the Chern forms of the universal vector bundles on $ \FF_{\rho}(E) $.
	Then the push-forward
	\[
	(\pi_{\rho})_{*} F\bigl(c_{\bullet}(\mathcal{E}_1,H_1),\dots,c_{\bullet}(\mathcal{E}_N,H_N)\bigr)
	\]
	is given by a universal (weighted) homogeneous polynomial formally evaluated in the Chern forms of $ (E,h) $.
\end{proposition}

\begin{remark}
	In order to prove the proposition we follow the steps in the proof of \cite[Proposition~{3.1}]{DF20} coupled with Theorem~\ref{thm: curvature of natural bundles} and the results obtained in Section~\ref{sect: intrinsic curvature vector bundle}.
	However, unlike \cite[Proposition~{3.1}]{DF20} where explicit computations are given, here we provide general remarks on the key concepts that allow us to prove the theorem, since, in the end, the computations are similar and explicit formul{\ae} are not strictly necessary.
\end{remark}

\begin{proof}[Proof of Proposition~\ref{lem: pol homogeneous general}]
	In the notation of Section~\ref{sect: intrinsic curvature vector bundle}, fix for a moment two indices $ 0 \le \ell < l \le m $ which identify a universal vector bundle $ \mathcal{E} $ over $ \FF_{\rho}(E) $, equipped with the Hermitian metric $ H $ induced by $ h $.
	
	By Formula~\eqref{eq: curvature of universal is sum of two tensors} we know that the curvature of $ (\mathcal{E},H) $ can be written as the sum of a horizontal tensor plus a vertical one (in the sense of Section~\ref{sect: intrinsic curvature vector bundle}).
	We denote them by $ \theta $ and $ \omega $ respectively, by ignoring the indices.
	
	For $ j = 1,\dots,\operatorname{rk} \mathcal{E} $, recall the well-known equality from linear algebra
	\begin{align*}
		c_{j}(\mathcal{E},H) &= \operatorname{tr}_{\operatorname{End}(\Lambda^{j}\mathcal{E})} \left( {\bigwedge}^{j} \frac{i}{2\pi} \Theta(\mathcal{E},H) \right) \\
		&= \frac{1}{j !} \det \begin{pmatrix}
			\operatorname{tr}(\theta + \omega) & j - 1 & 0 & \cdots &  \\
			\operatorname{tr} (\theta + \omega)^2 & \operatorname{tr}(\theta + \omega) & j - 2 & \cdots &  \\
			\vdots & \vdots &  & \ddots & \vdots \\
			\operatorname{tr}(\theta + \omega)^{j - 1} & \operatorname{tr}(\theta + \omega)^{j - 2} &  & \cdots & 1 \\
			\operatorname{tr}(\theta + \omega)^{j} & \operatorname{tr}(\theta + \omega)^{j - 1} &  & \cdots & \operatorname{tr}(\theta + \omega)
		\end{pmatrix}
	\end{align*}
	where we have used Formula~\eqref{eq: curvature of universal is sum of two tensors}.
	
	Now, fix a point $ x \in X $ and let $ \mathbf{f} $ be any flag in the fiber $ \FF_{\rho}(E_x) $ given by a unitary basis $ (v_1,\dots,v_r) $ of $ E_x $.
	Fixed a local normal frame $ (e_1,\dots,e_r) $ for $ E $ centered at $ x \in X $, we write $ v_{\lambda} = \sum_{\nu} v_{\lambda}^{\nu} \,e_{\nu}(x) $, where the coordinates $ v_{\lambda}^{\nu} $'s are as in Lemma~\ref{lem:blocksgeneral}.
	Moreover, we can write $ \theta $ (resp. $ \omega $) in matrix form with respect to the frame $ \big( e_1(x),\dots,e_r(x) \big) $ as $ (\theta_{pq}) $ (resp. $ \omega_{tu} $), where $ p,q,t,u = 1 , \dots, \operatorname{rk} \mathcal{E} $.
	
	Therefore, it is clear that each entry $ \operatorname{tr}(\theta + \omega)^{j} $ is a polynomial in $ \theta_{pq} $'s and $ \omega_{tu} $'s, whose coefficients depend only on the rank of $ \mathcal{E} $ and of the indices $ 0 \le \ell < l \le m $ defining the bundle $ \mathcal{E} $.
	Consequently, by the chain of equalities above, $ c_{j}(\mathcal{E},H) $ is a polynomial of type $ P^{j}(\theta_{pq}, \omega_{tu}) $ whose coefficients are universal.
	
	Now, Formula~\ref{eq: horizontal curvature explicit} gives us an explicit expression of $ \theta $ for which, w.r.t. the frame $ \big( e_1(x),\dots,e_r(x) \big) $, we get
	\begin{equation}\label{eq:horizontal part is a polynomial universal}
		\theta_{pq}(x,\mathbf{f}) =  \frac{i}{2\pi} \sum_{\alpha,\beta=1}^{r} Q_{\alpha\beta}(v,\bar{v}) \, \Theta_{\beta \alpha}(x) ,
	\end{equation}
	where we have denoted by $ v $ the vector of coordinates $ v_{\lambda}^{\nu} $'s.
	Observe that the polynomials $ Q_{\alpha\beta} $'s have universal coefficients by Formula~\ref{eq: horizontal curvature explicit}.
	
	In order to compute the push-forward, by linearity we can of course suppose that the given polynomial $ F $ is just a monomial.
	Thus, it is sufficient to show that the push-forward through $ \pi_{\rho} $ of
	\begin{equation}\label{eq:monomial to be pushed}
		\bigwedge_{s=1}^{N} \bigwedge_{j_s=1}^{r_s} c_{j_s}(\mathcal{E}_s,H_s)^{\wedge \ell_{j_s}} = \bigwedge_{s=1}^{N} \bigwedge_{j_s=1}^{r_s} {P_{s}^{j_s}(\theta_{pq}^{s}, \omega_{tu}^{s})}^{\wedge \ell_{j_s}}
	\end{equation}
	is a polynomial in the entries $\Theta_{\beta\alpha}$'s of the matrix associated to the curvature $\Theta(E,h)_x$ with respect to the frame $ (e_1,\dots,e_r )$.
	
	We have observed in Formula~\eqref{eq:horizontal part is a polynomial universal} that $ \theta_{pq}^{s} $ has universal polynomials in $ v $ and $ \bar{v} $ as coefficients.
	Therefore, coupling this with the universality of the $ P^j $'s, after performing all the wedge powers and products in the right hand side of Formula~\eqref{eq:monomial to be pushed}, we can rewrite the latter as a polynomial $ A $ in the $ \Theta_{\beta\alpha} $'s and $ \omega_{tu}^{s} $'s.
	Moreover, since all the polynomials involved in the right hand side of Formula~\eqref{eq:monomial to be pushed} are universal, and we are performing wedge powers and products of them, we deduce that the coefficients of $ A $ are universal polynomials in $ v $ and $ \bar{v} $.
	
	We recall that the $\Theta_{\beta\alpha}$'s only depend on the point $x$, while the $v_\lambda^\nu$'s can be seen, by a slight abuse of notation, as variables of integration even if they have to be understood modulo the action of
	\[
	U(r - \rho_{m-1}) \times U(\rho_{m-1} - \rho_{m-2}) \times \dots \times U(\rho_{1}) \ \subset \ U(r)
	\]
	which corresponds to the homogeneous presentation of the flag manifold as
	\[
	{\textstyle U(r)}\Big/_{\textstyle U(\rho_{1}) \times \dots \times U(r - \rho_{m-1})} \ .
	\]
	Moreover, only the $ \omega_{tu}^{s} $'s contain the vertical differentials which can be integrated along the fibers of the submersion $ \pi_{\rho} $.
	
	By applying $ (\pi_{\rho})_{*} $ to Formula~\eqref{eq:monomial to be pushed}, we have expressed
	$$ (\pi_{\rho})_{*} F\bigl(c_{\bullet}(\mathcal{E}_1,H_1),\dots,c_{\bullet}(\mathcal{E}_N,H_N)\bigr) $$
	as a polynomial in the $\Theta_{\beta\alpha}$'s, whose coefficients are integrals over $ \FF_{\rho}(E_x) $ of universal polynomials in the variables $ v_{\lambda}^{\nu} $'s, and the volume forms are products of the $ \omega_{tu}^{s} $'s.
	These integrals can be computed over the flag manifold $ \FF_{\rho}(\C^r) $, by applying a unitary transformation from $ (E_x,h_x) $ onto $ \C^r $ equipped with the standard metric.
	This concludes the proof.
\end{proof}

Now we see how to explicitly compute the push-forward
\begin{equation*}
	(\pi_{\rho})_{*} F\bigl(c_{\bullet}(\mathcal{E}_1,H_1),\dots,c_{\bullet}(\mathcal{E}_N,H_N)\bigr) \in \mathcal{A}^{k,k}(X) .
\end{equation*}
At the cohomology level, we formally evaluate $ F $ in the Chern classes of $ \mathcal{E}_1, \dots, \mathcal{E}_N $.
This gives a class $ F\bigl(c_{\bullet}(\mathcal{E}_1),\dots,c_{\bullet}(\mathcal{E}_N)\bigr) $ in the group $ H^{2(d_\rho + k)}\bigl(\FF_{\rho}(E)\bigr) $, which we write as a polynomial $ \tilde{F} $ in terms of the Chern roots $ \xi_1,\dots,\xi_r $ of $ \pi_{\rho}^{*}{E^{\vee}} $.
Thus, we have an equality
\[
F\bigl(c_{\bullet}(\mathcal{E}_1),\dots,c_{\bullet}(\mathcal{E}_N)\bigr) = \tilde{F}(\xi_{1},\dots,\xi_{r}) .
\]
By the known universal Gysin formul{\ae} for flag bundles at the level of cohomology, we know that there is a universal weighted homogeneous polynomial $ \Phi $ of degree $ 2k $ such that
\begin{equation}\label{eq: push at the class level, general form}
	(\pi_{\rho})_{*} F\bigl(c_{\bullet}(\mathcal{E}_1),\dots,c_{\bullet}(\mathcal{E}_N)\bigr) = \Phi\bigl(c_1(E),\dots,c_r(E)\bigr) .
\end{equation}

The next result, which is the main theorem of the paper, completely translates Formula~\eqref{eq: push at the class level, general form} at the level of differential forms and generalizes the pointwise push-forward formul{\ae} for flag bundles given in the previous literature.

\begin{theorem}\label{thm: darondeau-pragacz general for differential forms}
	Let $ (E,h) $ be a rank $ r $ Hermitian holomorphic vector bundle over a complex manifold $ X $ of dimension $ n $, and let $ F $ be a complex homogeneous polynomial of degree $ d_{\rho}+k $ in $ r_1 + \dots + r_N $ variables.
	Then, we have the equality
	\begin{equation}\label{eq: Gysin formula universal vector bundles}
		(\pi_{\rho})_{*} F\bigl(c_{\bullet}(\mathcal{E}_1,H_1),\dots,c_{\bullet}(\mathcal{E}_N,H_N)\bigr) = \Phi\big(c_1(E,h),\dots,c_r(E,h)\big) .
	\end{equation}
\end{theorem}
\begin{proof}
	The difference
	\[
	(\pi_{\rho})_{*} F\bigl(c_{\bullet}(\mathcal{E}_1,H_1),\dots,c_{\bullet}(\mathcal{E}_N,H_N)\bigr) - \Phi\big(c_{\bullet}(E,h)\big)
	\]
	is of course an exact global $ (k,k) $-form on $ X $.
	By Proposition~\ref{lem: pol homogeneous general},
	$$ (\pi_{\rho})_{*} F\bigl(c_{\bullet}(\mathcal{E}_1,H_1),\dots,c_{\bullet}(\mathcal{E}_N,H_N)\bigr) $$
	is a universal homogeneous polynomial of weighted degree $ 2k $ in the Chern forms of $ (E,h) $.
	Hence, the previous difference can be written as a complex weighted homogeneous polynomial $ G $ in the Chern forms of $ (E,h) $, whose coefficients depend only upon $ r,n $ and $ F $, not on the vector bundle $ E $.
	
	At this point, we claim that $ G \equiv 0 $ since we can proceed in the same way as \cite[Proposition 3.1]{Gul12} and \cite[Theorem 3.5]{DF20}.
	Indeed, for $ m_1,\dots,m_r $ positive integers, we formally evaluate $ G $ on the Chern forms of the totally split, rank $ r $ vector bundle $ ({A}^{\otimes m_1} \oplus \dots \oplus {A}^{\otimes m_r}, \omega) $, where $ A $ is an ample line bundle on a projective $ n $-manifold and $ \omega $ is the natural metric induced by a metric on $ A $ with positive curvature.
	The claim follows from the universality of the polynomial $ G $. 
\end{proof}

\section{Applications}
This section is devoted to two applications of Theorem~\ref{thm: darondeau-pragacz general for differential forms}.
The former is a refinement on the level of differential forms of the Jacobi--Trudi identity and the latter proves the positivity of several differential forms thus giving new evidences towards Griffiths' conjecture on positive characteristic forms.

\subsection{Schur forms as push-forwards}
Here we present an explicit formula that allows us to obtain every Schur form as some push-forward from the complete flag bundle.

First, we recall some notation from \cite{DP17,Fag20}.
If $ \sigma = (\sigma_1,\dots,\sigma_k) $ is a $ k $-uple of integers, we denote by $ s_{\sigma} $ the \emph{generalized Schur polynomial} defined as $ \det \bigl( s_{\sigma_i + j - i} \bigr)_{1 \le i,j \le k} $, where the $ s_\ell $'s are the Segre polynomials.

Let $ (E,h) $ be a Hermitian holomorphic rank $ r $ vector bundle over a complex $ n $-manifold $ X $.
The polynomial $ s_\sigma $ formally evaluated in the Chern forms of $ (E,h) $ (resp. Chern classes of $ E $) is denoted by $ s_\sigma(E,h) $ (resp. $ s_\sigma(E) $) and is called the \emph{generalized Schur form} (resp. \emph{class}) associated to $ \sigma $.
This terminology is inspired by the identity~\eqref{eq:relation schur vs generalized schur} below.

Now, we point out how to obtain an explicit expression for the right hand side of Formula~\eqref{eq: Gysin formula universal vector bundles}.
Taken a homogeneous polynomial $F$, in the notation of Theorem~\ref{thm: darondeau-pragacz general for differential forms}, we recall that there exists a polynomial in the virtual Chern roots of $ \pi_{\rho}^{*}E^{\vee} $ called
$$
\tilde{F}(\xi_1,\dots,\xi_r) = \sum_{|\lambda| = d_{\rho}+k} b_{\lambda} \xi_{1}^{\lambda_1} \cdots \xi_{r}^{\lambda_r} ,
$$
which has the appropriate symmetries which ensure that
\begin{equation*}
	\tilde{F}(\xi_1,\dots,\xi_r) = F\bigl(c_{\bullet}(\mathcal{E}_1),\dots,c_{\bullet}(\mathcal{E}_N)\bigr) \in H^{2(d_{\rho} + k)}\bigl(\FF_{\rho}(E)\bigr) .
\end{equation*}
By \cite[Proposition 1.13]{Fag20} we have that
\begin{equation}\label{eq: darondeau-pragacz classes schur}
	(\pi_{\rho})_{*} \tilde{F}(\xi_1,\dots,\xi_r) = \sum_{|\lambda| = d_{\rho}+k} b_{\lambda} s_{(\lambda-\nu)^{\leftarrow}}(E)
\end{equation}
where $ \nu $ is the non-decreasing sequence of integers determined by $ \rho $ as:
\begin{equation}\label{eq: condition on index generalized schur class}
	\nu_i = r - \rho_{\ell} \quad \text{for } r - \rho_{\ell} < i \le r - \rho_{\ell-1} ,
\end{equation}
the notation $ (\sigma_1,\dots,\sigma_r)^{\leftarrow} $ stands for $ (\sigma_r,\dots,\sigma_1) $ and the difference of $ \lambda - \nu $ is defined componentwise.

Due to Theorem~\ref{thm: darondeau-pragacz general for differential forms}, we can now extend \cite[Formula~{(6)}]{Fag20} to any possible
\[
F\bigl(c_{\bullet}(\mathcal{E}_1,H_1),\dots,c_{\bullet}(\mathcal{E}_N,H_N)\bigr)
\]
not only those of the form
\[
F\bigl(\dots,c_1(U_{\rho,\ell+1}/U_{\rho,\ell},h_{\rho,(\ell+1,\ell)}),\dots\bigr).
\]
Therefore, by Formula~\ref{eq: darondeau-pragacz classes schur} and Theorem~\ref{thm: darondeau-pragacz general for differential forms} it follows the identity
\begin{equation}\label{eq: darondeau-pragacz forms schur}
	(\pi_{\rho})_{*} F\bigl(c_{\bullet}(\mathcal{E}_1,H_1),\dots,c_{\bullet}(\mathcal{E}_N,H_N)\bigr) = \sum_{|\lambda| = d_\rho + k} b_{\lambda} s_{(\lambda-\nu)^{\leftarrow}}(E,h) ,
\end{equation}
where $ \nu $ and the $ b_{\lambda} $'s are as descibed above.
Obviously, the right hand side of Formula~\eqref{eq: darondeau-pragacz forms schur} equals the form $ \Phi\big(c_{\bullet}(E,h)\big) $ given by Theorem~\ref{thm: darondeau-pragacz general for differential forms}.

\medskip
Let $ \sigma $ be a partition  in $ \Lambda(k,r) $.
From the Jacobi--Trudi identities for Schur polynomials we know that
\begin{equation}\label{eq:relation schur vs generalized schur}
	S_{\sigma}(E,h) = (-1)^{|\sigma|} s_{\sigma^{\prime}}(E,h)
\end{equation}
where $ \sigma^{\prime} $ is the conjugate partition of $ \sigma $, obtained through the transposition of the Young diagram of $ \sigma $.

\begin{remark}\label{rem: change sequences of schur}
	Observe that we can canonically associate to the conjugate partition $ \sigma^{\prime} \in \Lambda(k,r) \subset \N^k $ an element $ \tilde{\sigma} = (\tilde{\sigma}_1,\dots,\tilde{\sigma}_r)$ of $ \N^r $ as follows.
	If
	\begin{itemize}
		\item $ k < r $, then we set $ \tilde{\sigma} = (\sigma_1^{\prime},\dots,\sigma_{k}^{\prime},\sigma_{k + 1}^{\prime}=0,\dots,\sigma_{r}^{\prime}=0) $;
		\item $ k = r $, then $ \tilde{\sigma} $ is $ \sigma^{\prime} $;
		\item $ k > r $, then, since $ r \ge \sigma_1 \ge \dots \ge \sigma_{k} \ge 0 $, by definition $ \sigma^{\prime} $ must be of the form $ (\sigma_1^{\prime},\dots,\sigma_{r}^{\prime},0,\dots,0) $.
		Therefore we set $ \tilde{\sigma} = (\sigma_1^{\prime},\dots,\sigma_{r}^{\prime}) $.
	\end{itemize}
	In all of the three cases above, it follows from the definition that $ s_{\tilde{\sigma}}(E,h) = s_{\sigma^{\prime}}(E,h) $.
\end{remark}
The next result shows that the Schur form $ S_{\sigma}(E,h) $ can be obtained as a push-forward from the complete flag bundle associated to $ E $.
We follow the notation introduced in Remark~\ref{rem: change sequences of schur}.
\begin{proposition}\label{prop: schur forms as push-forwards}
	Let $ \sigma $ be a partition in $ \Lambda(k,r) $ and let $ \lambda \in \Z^{r} $ be the sequence whose $ j $-th element is $ \lambda_j = \tilde{\sigma}_{r-j+1} + j-1 $.
	Then
	\begin{equation*}\label{eq: push to obtain the auxiliary schur form}
		\pi_{*} \big[ (-1)^{|\lambda|+|\sigma|} \Xi_1^{\wedge \tilde{\sigma}_r} \wedge \dots \wedge \Xi_r^{\wedge (\tilde{\sigma}_{1} + r - 1)} \big] = S_{\sigma}(E,h)
	\end{equation*}
	where $ \pi $ is the natural projection of the complete flag bundle $ \FF(E) $, and the $ \Xi_j $'s are the first Chern forms of its tautological successive quotients equipped with the induced metrics.
\end{proposition}
\begin{proof}
	The sequence $ \lambda $ is defined in such a way that $ \tilde{\sigma} = (\lambda - \nu)^{\leftarrow} $, where $ \nu = (0,1,\dots,r-1)$ is as in Formula~\eqref{eq: condition on index generalized schur class}.
	Therefore, by Formula~\eqref{eq: darondeau-pragacz forms schur} it follows that
	\[
	\pi_{*} \bigl[ (-\Xi_1)^{\lambda_1} \wedge \dots \wedge (-\Xi_r)^{\lambda_r} \bigr] = s_{(\lambda - \nu)^{\leftarrow}}(E,h) = (-1)^{|\sigma|} S_{\sigma}(E,h)
	\]
	where the last equality follows from Formula~\eqref{eq:relation schur vs generalized schur} and Remark~\ref{rem: change sequences of schur}.
\end{proof}

\begin{remark}
	Observe that the Chern root $ \xi_\ell = - c_1(U_{r-\ell+1}/U_{r-\ell}) $ of the complete flag bundle is not virtual and is represented by the form $ -\Xi_\ell $.
\end{remark}

\begin{remark}
	Proposition~\ref{prop: schur forms as push-forwards} is an alternative expression of the Jacobi--Trudi identity for differential forms given in \cite[Theorem~{A.2}]{Fin20}.
	Therefore, these two results are both a special case of Theorem~\ref{thm: darondeau-pragacz general for differential forms}.
	Note also that to obtain Proposition~\ref{prop: schur forms as push-forwards}, Theorem~\ref{thm: darondeau-pragacz general for differential forms} in its generality is not necessary, but it is sufficient to use \cite[Theorem~{3.5}]{DF20}.
\end{remark}

\subsection{Positive characteristic forms for Griffiths positive vector bundles}
In this section we use Theorem~\ref{thm: darondeau-pragacz general for differential forms} to prove the positivity of a family of differential forms and their belonging to the Schur cone of a Griffiths semipositive vector bundle.
Therefore, in the same spirit of \cite{Gul12,DF20}, we give new evidence to Griffiths' conjecture about the positivity of Chern--Weil forms.

Assume that the Hermitian holomorphic vector bundle $ (E,h) \to X $ is now {Griffiths semipositive}.

For $ s = 1,\dots,r-1 $ let $ \GG_{s}(E) $ be the Grassmann bundle associated to $ E $ and $ \gamma_{s} $ its tautological vector bundle.
By a slight abuse of notation we denote by $ p $ the natural projection from a Grassmann bundle to $ X $, ignoring the index $ s $.
Let $ Q_s = p^{*}E / \gamma_{s} $ be the universal quotient bundle of rank $ r-s $ endowed with the quotient metric $ h_s $ induced by $ h $.
Recall that, in our notation,
\begin{itemize}
	\item if $ s = 1 $, then $ \GG_{1}(E) = \PP(E) $ is the projective bundle of lines in $ E $ and $ Q_1 = p^{*}E / \mathcal{O}_{\PP(E)}(-1) $;
	\item if $ s = r-1 $, then $ \GG_{r-1}(E) = \PP(E^{\vee}) $ is the projective bundle of hyperplanes in $ E $ and $ Q_{r-1} = \mathcal{O}_{\PP(E^{\vee})}(1) $.
\end{itemize}

Being a quotient of a Griffiths semipositive vector bundle $ (Q_s,h_s) $ is Griffiths semipositive (see \cite[{(6.10)}~Proposition]{Dem01}).

\begin{theorem}\label{thm: cone of positive forms}
	For each $ \alpha, \beta $ such that
	\begin{equation*}
		0 \le \beta \le 2 \quad \text{and} \quad s(r-s) \le \alpha + 2\beta \le n + s(r-s)
	\end{equation*}
	the closed differential form
	\begin{equation*}
		p_{*} \left[ {c_1(Q_s,h_s)}^{\wedge \alpha} \wedge {c_2(Q_s,h_s)}^{\wedge \beta} \right]
	\end{equation*}
	is a positive form on $ X $ belonging to the positive convex cone spanned by the Schur forms of $ (E,h) $.
	Moreover, the explicit expression of the push-forward can be obtained by applying, for instance, Formula~\eqref{eq: darondeau-pragacz forms schur}.
\end{theorem}
\begin{proof}
	To simplify the notation, set $ c_1 := c_1(Q_s,h_s) $ and $ c_2 := c_2(Q_s,h_s) $.
	First, recall that the positivity of $ c_1 $ follows from the definition of Griffiths semipositivity.
	Regarding the positivity of $ c_2 $ we refer to \cite[Appendix to~\S 5.(b)]{Gri69}, \cite[6]{Gul06} and \cite[Theorem~{2.2}]{Fag20}.
	Observe that the condition $ 0 \le \beta \le 2 $ is needed as the positivity of $ c_2^{\wedge \beta} $ is not known if $ \beta > 2 $, while the positivity of $ c_2^{\wedge 2} $ is the content of \cite[Theorem~1]{BP13}.
	
	Call $ P $ the polynomial obtained by applying Theorem~\ref{thm: darondeau-pragacz general for differential forms} to $ p_{*} ( c_1^{\wedge \alpha} \wedge  c_2^{\wedge \beta} )$ and recall that the coefficients of $ P $ do not depend on the vector bundle $ (E,h) $.
	They depend only on $ n, r $ and of course the shape of the monomial $ x^{\alpha}  y^{\beta} $.
	
	Also, we point out that the polynomial $ P $ must be \emph{pointwise positive for Griffiths semipositive vector bundles}.
	This means that for each Griffiths semipositive vector bundle $ (\mathcal{V},h_{\mathcal{V}}) $ of rank $ r $ over a complex manifold $ Z $, the form $ P\big( c_{\bullet}(\mathcal{V},h_{\mathcal{V}}) \big) $ is positive as a differential form.
	This follows from Theorem~\ref{thm: darondeau-pragacz general for differential forms} as $ P\big( c_{\bullet}(\mathcal{V},h_{\mathcal{V}}) \big) $ is obtained, by definition, as the push-forward of the positive form
	\[
	{c_1(\mathcal{Q}_s,h_{\mathcal{Q}_s})}^{\wedge \alpha} \wedge {c_2(\mathcal{Q}_s,h_{\mathcal{Q}_s})}^{\wedge \beta} ,
	\]
	where $ \mathcal{Q}_s $ is the universal quotient bundle over the Grassmann bundle $ \GG_{s}(\mathcal{V}) $, and $ h_{\mathcal{Q}_s} $ is the quotient metric induced by $ h_{\mathcal{V}} $.
	
	Moreover, we recall that $ P $ must be a non-negative combination of Schur polynomials.
	This follows from \cite[Proposition~{3.4}]{FL83} where it is shown that if $ P $ is not in the cone spanned by Schur polynomials then there exists a smooth projective $ n $-manifold and a dual Nakano positive vector bundle of rank $ r $ over it, such that the integral over the manifold of $ P $ evaluated in the Chern classes of the bundle gives a negative number (we refer to \cite[Appendix~A]{FL83}, \cite[Remark~{4.4}]{DF20}).
	This contradicts the fact that $ P $ is pointwise positive for Griffiths semipositive vector bundles, as the integral of a positive form gives a non-negative number.
\end{proof}

\begin{remark}
	By Theorem~\ref{thm: cone of positive forms}, Griffiths' conjecture on positive characteristic forms is thus confirmed for the polynomials in the Chern forms of $ (E,h) $ which are in the positive convex cone $ \mathcal{G}(E,h) $ spanned by all possible push-forwards as in Theorem~\ref{thm: cone of positive forms}.
	Observe however that the cone $ \mathcal{G}(E,h) $ is not stable under wedge product as the forms we push-forward are in general not strongly (resp. Hermitian) positive (see \cite[Section~{1.1}]{Fag20}).
\end{remark}

Now, we want to point out that the cone $ \mathcal{G}(E,h) $ and the cone $ \mathcal{F}(E,h) $ obtained in \cite[Theorem 4.8]{DF20} are different, and one is not contained in the other.
For this part, we have used both Formula~\eqref{eq: darondeau-pragacz forms schur} and Darondeau--Pragacz push-forward formula \cite[Proposition~{1.2}]{DP17}.
See also \cite[Formula~{(3.3)}]{DF20}.

If $ E $ has rank $ 3 $ then $ c_2 $ is in the cone $ \mathcal{G}(E,h) $ but not in $ \mathcal{F}(E,h) $.
Indeed, with respect to the Schur basis
\[
\left[ S_{(2,0)} = c_2, \  S_{(1,1)} = c_1^2 - c_2 \right]
\]
the $ (2,2) $-forms belonging to $ \mathcal{F}(E,h) $ are given by positive linear combinations of forms with coordinates
\begin{enumerate}[(i)]
	\item $ [2a(a^3 -3ab^2 +2b^3), 3a^4 -4a^3b + b^4] $ if $ \rho = (0,1,3) $, for each $ a > b \ge 0 $;\label{item:proj}
	\item $ [2b(2a^3 -3a^2b +b^3), a^4 -4ab^3 + 3b^4] $ if $ \rho = (0,2,3) $, for each $ a > b \ge 0 $;\label{item:hyper}
	\item $ [10( a^2b^2(a-b) -a^2c^2(a-c) +b^2c^2(b-c) ), 5( ab(a^3-b^3) -ac(a^3-c^3) +bc(b^3-c^3) )] $ if $ \rho = (0,1,2,3) $, for each $ a > b > c \ge 0 $.\label{item:complete}
\end{enumerate}
By a direct computation, it follows that $ c_2 $ can not be obtained as a positive linear combination of~\eqref{item:proj},~\eqref{item:hyper} and~\eqref{item:complete} with the parameters satisfying those conditions.
Therefore, $ c_2 \notin \mathcal{F}(E,h) $.
On the other hand, $ c_2 \in \mathcal{G}(E,h) $ as it equals the push-forward of $ c_2(Q_1,h_1)^2 $ from the projective bundle $ \PP(E) $.

With a similar argument one can also show that the Schur form $ c_1c_2 - c_3 $ is in $  \mathcal{G}(E,h) $ but not in $ \mathcal{F}(E,h) $.

Conversely, we see that if $ E $ has rank $ 2 $, then we can find forms belonging to $ \mathcal{F}(E,h) $ but not to $ \mathcal{G}(E,h) $.
Indeed, for $ r=2 $ we have just one (proper) flag bundle, namely $ \PP(E) $ and, looking at $ (2,2) $-forms, Theorem~\ref{thm: cone of positive forms} only gives the positivity of the non-negative multiples of $ S_{(1,1)} = c_1^2 - c_2 $.
Instead, the $ (2,2) $-forms in $ \mathcal{F}(E,h) $ consist of positive linear combinations of forms with Schur coordinates
\[
[3ab(a-b), a^3 - b^3], \ \text{for each} \  a > b \ge 0
\]
and by choosing $ b>0 $ we are done.

Finally, we note that in any rank the positive convex cones $ \mathcal{F}(E,h) $ and $ \mathcal{G}(E,h) $ intersect (at least) in the Schur forms $ S_{(1,\dots,1)}(E,h) $'s (which are the signed Segre forms).
This follows from \cite[Section~{4.1.1}]{DF20} and from Theorem~\ref{thm: cone of positive forms} by choosing $ s=r-1 $.

\subsubsection{Examples}
We conclude the paper by giving some concrete examples of differential forms whose positivity is due to Theorem~\ref{thm: cone of positive forms}.
In these examples, we suppose that $ r = n = 4 $.

In order to simplify the notation, we denote by $ c_1,\dots,c_r $ the Chern forms of $ (E,h) $ and by $ s_1,\dots,s_n $ its Segre forms.
We are also allowed to omit the symbol $ \wedge $ of wedge product of forms.

Set $ s=1 $, and consider $ (Q_1,h_1) \to \PP(E) $.
The push-forward of the positive form $ {c_1(Q_1,h_1)}^{2}  {c_2(Q_1,h_1)}^{2} $ gives the differential form
\begin{align*}
c_1^3 + 2c_1c_2 - c_3 &= -2 s_1^3 + s_3 \\
&= 2S_{(3,0,0)} + 4S_{(2,1,0)} + S_{(1,1,1)} ,
\end{align*}
while the push-forward of $ {c_1(Q_1,h_1)}^{3}  {c_2(Q_1,h_1)}^{2} $ gives
\begin{align*}
	c_1^4 + 3c_1^2c_2 - 3c_1c_3 -c_4 &= 6s_1^2s_2 - 5s_1s_3 -s_2^2 +s_4 \\
	&= 6S_{(3,1,0,0)} +5S_{(2,2,0,0)} + 6S_{(2,1,1,0)} +S_{(1,1,1,1)} .
\end{align*}

Set $ s=2 $, and consider $ (Q_2,h_2) \to \GG_{2}(E) $.
The push-forward of the positive form $ {c_1(Q_2,h_2)}^{3}  {c_2(Q_2,h_2)}^{2} $ gives the differential form
\begin{align*}
	c_1^3 - c_3 &= -2 s_1s_2 + s_3 \\
	&= 2S_{(2,1,0)} + S_{(1,1,1)} ,
\end{align*}
while the push-forward of $ {c_1(Q_2,h_2)}^{4}  {c_2(Q_2,h_2)}^{2} $ gives
\begin{align*}
	c_1^4 - 3c_1c_3 + 2c_4 &= s_1s_3 +2s_2^2 -2s_4 \\
	&= 2S_{(2,2,0,0)} + 3S_{(2,1,1,0)} +S_{(1,1,1,1)} .
\end{align*}

Therefore, already in the cases of $ \PP(E) $ and $ \GG_{2}(E) $ we have new examples of positive forms.

% bibliography
%\cleardoublepage
%\phantomsection % Give this command only if hyperref is loaded
%\addcontentsline{toc}{chapter}{\bibname}
%\printbibliography
\bibliographystyle{amsalpha}
\bibliography{bibliography}{}

\end{document}